\documentclass[11pt,a4paper,reqno]{amsart}

\usepackage{color,latexsym,amsfonts,amssymb,amsmath,cite, amsthm, bbm}
\usepackage{fullpage}
\usepackage{fancyhdr}
\usepackage{graphicx}
\usepackage{tikz}
\usepackage{todonotes}
\usepackage{dsfont}
\usepackage{bbm}
\usepackage{mdwlist}
\usepackage{soul}

\providecommand{\N}{\mathbb{N}}
\providecommand{\R}{\mathbb{R}}
\providecommand{\B}{\mathbb{B}}

\providecommand{\Z}{{\mathbb{Z}}}

\renewcommand{\P}{\mathbb{P}}

\providecommand{\E}{\mathbb{E}}

\providecommand{\M}{\mathbb M}
\providecommand{\K}{\mathbb K}
\providecommand{\MM}{\mathbb M \times \mathbb M'}
\providecommand{\Mi}{M^{\ms{in}}}
\providecommand{\Mo}{M^{\ms{out}}}
\providecommand{\LL}{\mathbb L}
\providecommand{\NM}{\mathbb{N}_{\M}}
\providecommand{\NK}{\mathbb{N}_{\K}}
\providecommand{\NMM}{\mathbb{N}_{\MM}}
\providecommand{\T}{\mc{T}}
\providecommand{\Ta}{\mc{T}_{\ms{adm}}}

\providecommand{\shift}{S} 

\renewcommand{\d}{{\rm d}}
\providecommand{\ol}{\overline}

\providecommand{\opt}{\ms{opt}}
\providecommand{\es}{\emptyset}

\providecommand{\one}{\mathbbm{1}}

\providecommand{\mc}{\mathcal}
\providecommand{\T}{\mc T}

\providecommand{\Xis}{\Xi^{\e}}
\providecommand{\mmax}{\ms{max}}
\providecommand{\opt}{\ms{opt}}

\def\wt{\widetilde}

\newcommand{\ind}{\mathbbm{1}}

\providecommand{\g}{\gamma}
\providecommand{\gh}{\gamma_\xi}
\providecommand{\e}{\varepsilon}
\providecommand{\la}{\lambda}

\providecommand{\ms}{\mathsf}

\newcommand{\vM}{\check{M}}

\newcommand{\exclude}[1]{}
\newtheorem{theorem}{Theorem}[section]

\newtheorem{proposition}[theorem]{Proposition}

\newtheorem{lemma}[theorem]{Lemma}
\theoremstyle{definition}


\begin{document}

\author{Bart{\l}omiej B{\l}aszczyszyn}
\author{Christian Hirsch}
\address[Bart{\l}omiej B{\l}aszczyszyn]{Inria and DI ENS/PSL University, 2 rue Simone Iff, CS 42112, 75589 Paris Cedex 12, France}
\email{Bartek.Blaszczyszyn@ens.fr}
\address[Christian Hirsch]{Bernoulli Institute, University of Groningen, Nijenborgh 9, 9747 AG Groningen, The Netherlands}
\email{c.p.hirsch@rug.nl}

\title{Optimal stationary markings}

\date{\today}

\begin{abstract}
Many specific problems ranging from  theoretical probability to applications in statistical physics, combinatorial optimization and communications can be formulated as an  optimal tuning of local parameters
in large systems of interacting particles. Using the framework of stationary  point processes in the Euclidean space, we  pose it as a problem of an optimal stationary marking of a given stationary point process.
The quality of a given marking is evaluated in terms of scores  calculated in a covariant manner for all points in function of the proposed  marked configuration. In the absence of total order of the configurations of scores,  
we identify intensity-optimality and local optimality as two natural ways  for defining optimal stationary marking. 
We derive tightness and integrability conditions under which intensity-optimal markings exist and further stabilization conditions making them equivalent to locally optimal ones.  We present examples   motivating the proposed, general framework. Finally, we discuss various possible approaches  leading to  uniqueness results. 
\end{abstract}

\maketitle
\tableofcontents

\section{Introduction}

%
%
 This paper discusses a novel class of stationary markings of stationary point and particle processes. These markings are thought of as \emph{optimal} in the sense that they  maximize a score function calculated in a translation-covariant way  over all points of the point process. As a simple illustration, consider 
 a stationary germ-grain process and ask for its stationary thinning that maximizes the volume fraction of the retained grains under their non-overlapping (hard-core) constraint. It is an instance of an optimal binary marking, with the marks indicating whether or not the  grains are retained in the  thinned process. The score function is non-null only for the retained grains, returning their volumes provided no overlapping in the thinned process, and $-\infty$ otherwise. This is the volume-maximal hard-core thinning  considered  in~\cite{thinning}.
 Many other problems arising in various domains of science and engineering can be formulated as an optimal stationary marking with appropriate marks and score functions. Several examples will be presented in this paper.

%
%
Prior to developing algorithms for finding optimal solutions, in this paper we address a more fundamental question: What do we mean by an optimal marking of a stationary process? While in a finite domain, one can simply impose  that the sum of the scores evaluated for all points in the domain should be maximized, this procedure is ill-defined for stationary (hence infinite) systems.  Inspired by~\cite{thinning} and some specific problems considered in the literature, we consider the following two natural ways of defining an optimal marking in the stationary setting.
A stationary marking is:
\begin{itemize}
\item \emph{Intensity-optimal}
if it  maximizes the score-intensity, i.e.,  the expected score of the typical point or, equivalently, the expected aggregated scores of all points in a unit volume.
Existence of an  intensity-optimal marking is hence equivalent to the existence of a stationary marking attaining the supremum of the score-intensities
over all stationary markings. 
\item  \emph{Locally optimal} if, roughly speaking, one cannot achieve a net increase of the sum of scores by changing the marks of finitely many points only.
\end{itemize}

We address the following key questions regarding these two notions of optimality.
\begin{enumerate}
	\item	\emph{Do intensity- and locally optimal markings exist?}\\
	\item	\emph{Are  the two notions equivalent?}\\
	\item	\emph{Are these optimal markings  unique in distribution?}\\
\end{enumerate}
We provide general sufficient conditions for questions (1) and (2), and verify them in a variety of examples. On the other hand, as in Gibbsian systems, the question on uniqueness is of a different flavor, and we expect the answers to be model-specific. A relation between uniqueness and the absence of suitable continuum percolation was already noted in \cite{thinning}. In the present work, we highlight further mechanisms that may lead to uniqueness, for instance in  a rather specific scenario, allowing one to use the mass transport principle to represent the score-intensity. We also provide an example for non-uniqueness when the underlying space is a tree.

%
%
First original techniques for establishing existence and equivalence of intensity- and locally optimal markings have been developed in \cite{thinning} for the volume-maximal hard-core thinning  problem that served us as an illustration. However, new challenges arise for general optimal markings. Firstly,  while the binary mark space of the thinnings is compact, we need appropriate tightness conditions in the general case. Secondly, our framework makes it possible to incorporate non-local dependence of scores on marks.  To deal with such cases, in this paper we revisit the  concept of stabilization that  has already proven useful for establishing various limit results in stochastic geometry \cite{stab3, stab4}. In its original form, the stabilization requires, loosely speaking,  that changes to marks should only impact the scores up to a finite, possibly random range. We relax this condition by allowing infinitely long impact range, as long as the accumulated changes of the scores outside a finite, random range stay small. In addition to these extensions, we also describe an approximation method to deal with possibly discontinuous score functions.

%
%
The remaining part of the manuscript is organized as follows. First, in Section \ref{modelSec}, we introduce the concepts of intensity-optimal and locally optimal markings, state the main existence, equivalence and uniqueness results and provide a variety of examples.  Sections \ref{existSec} and \ref{equivSec} contain the proofs of the existence and equivalence results, respectively. Appendix \ref{appendix} contains a general Palm-convergence result that may be of independent interest.

    \section{Model definition and main results}
\label{modelSec}

\subsection{Definitions}
\label{defSec}
Throughout the manuscript, $\Phi = \{(X_i, M_i)\}_{i}$ shall always denote a given, simple, stationary
$\M$-marked point process, where $\M$ is some locally compact, second countable, Hausdorff space.
That is,
$\Phi$ is a random variable in the space $\NM$ of locally finite
subsets of $\R^d \times \M$ endowed with the evaluation
$\sigma$-algebra~\cite{pp1}, with any set $\{x\}\times\M$, $x \in \R^d$, containing at
most one element of $\Phi$. 
Moreover, $\Phi$ is stationary in the sense that its distribution $\mc L(\Phi)=\mc L(S_y\Phi)$ is invariant with respect to shifts $\shift_y$ pushing atoms $(x,m)\in\R^d\times\M$ by any vector $y\in \R^d$, $(x,m) \mapsto (x - y,m)$, while preserving marks. We call $\Phi$ the {\em
base process} and denote by
\begin{equation}\label{e.intensity}
  \lambda = \E[\Phi ([0,1]^d \times \M)]
\end{equation}
its intensity, which is assumed to be finite and positive throughout the manuscript.

We consider a setting, where the base process $\Phi$ is enhanced by additional marks from 
a locally compact, second countable, Hausdorff space~$\M'$, which is in general different from $\M$. 
In the example of the germ-grain model, $\M$ is the space of compact, convex grains, while binary marks in $\M' = \{0, 1\}$ indicate whether or not the grains are retained in the thinned process.

In this regard,
we extend $\NM$ to $\NMM$ ---
the space of locally finite
subsets of $\R^d \times \M\times \M'$ endowed with the respective evaluation 
$\sigma$-algebra and say that 
$\psi\in\NMM$ is a $\M'$-{\em marking}
(or just {\em marking} for short) of 
the base configuration $\varphi\in\NM$ if 
$\varphi=\pi(\psi)$, where 
$\pi$ is the canonical projection of $\NMM$ to $\NM$. We naturally extend the shifts $\shift_y$ making them transport doubly-marked atoms: $(x,m,m') \mapsto (x - y,m,m')$
while preserving both marks.

An $(\MM)$-marked point process $\Psi = \{(X_i, M_i, M_i')\}_{i}$ is 
stationary if $\mc L(\Psi)=\mc L(S_y\Psi)$ for all $y\in\R^d$.
$\Psi$ is a \emph{stationary $\M'$-marking of $\Phi$}, in symbols $\Psi \in \T(\Phi)$, if moreover the image of its law $\mc L(\Psi)$ under the canonical projection $\pi$ coincides with $\mc L(\Phi)$; $\mc L(\pi(\Psi))=\mc L(\Phi)$.~\footnote{\label{fn.Probability-space} Note that a marking $\Psi$ is not necessarily a deterministic function of $\Phi$
and we allow it to be defined 
on some new  probability space.
If this is the case, it is customary to "re-define" the base process on this new space as the projection $\Phi:=\pi(\Psi)$ 
of $\Psi$ and, slightly abusing the notation, use the same expectation $\E$ as in~\eqref{e.intensity} regarding the new probability space. We assume this convention throughout the paper.}


To measure the quality of markings, we use a \emph{score function}
$\xi$ mapping from the set 
$$\NMM^0 := \{\psi\in\NMM:\psi(\{o\}\times \M \times \M') = 1\}$$
of all realizations of $\Psi$ having a point at the origin $o\in\R^d$
 either to $[0, \infty) \cup
 \{-\infty\}$ or to $[-\infty, 0]$.
 For $\psi\in\NMM^0$, $\xi(\psi)$
 represents the score of the point at the origin~$o$
 and depends, possibly, on the entire configuration of points with their double marks in $\psi$. The 
 score value $-\infty$ represents an inadmissible marking (for example, leading to some overlapping of grains when a hard-core thinning is required).
Using the covariance property, we extend the evaluation of the scores to all points $(x_i,m_i,m_i')\in\psi$ of any configuration $\psi\in\NMM$
$$\xi_i := \xi(\shift_{x_i}\psi).$$
Note, $\xi_i$ represents yet another stationary marking (enhancement) of the base process, this time a deterministic function of $\Psi$.

 A stationary marking $\Psi = \{(X_i, M_i, M_i')\}_{i}$ is \emph{admissible}, in symbols $\Psi \in \Ta(\Phi)$, if almost surely $\xi(\shift_{X_i}\Psi) \ne -\infty$ for all $i \ge 1$.

For a given base process $\Phi$, marking space~$\M'$, and a score function~$\xi$, we want to 
investigate (admissible) stationary
markings $\Psi$ maximizing the scores $\xi_i$ in one of the two meanings
described in what follows.
\subsubsection{Intensity-optimal marking}
For a given stationary marking~$\Psi$ of $\Phi$ and a score function~$\xi$, the \emph{$\xi$-intensity}
\begin{equation}\label{eq:xi-intensity-stat}
\gh(\Psi) := \E \Big[\sum_{\substack{X_i \in \Psi \cap
 [0,1]^d}}\xi(\shift_{X_i}\Psi)\Big]\,.
\end{equation}
is the expected total score of the points in a unit cube.
Here, and in what follows we often abuse notation and write $\Psi \cap [0, 1]^d$ for the projection of $\Psi \cap ([0, 1]^d \times \MM)$ to $\R^d$.

Equivalently,
\begin{equation}
\label{eq:xi-intensity-palm}
\gh(\Psi) = \lambda \E^0[\xi(\Psi)] =\lambda \E^0[\xi_0],
\end{equation}
where
$\E^0[\xi_0]$ is the expected $\xi$-score of the typical 
point of $\Phi$, that is of the point located at the origin
$X_0 = o$ under the Palm probability $\P^0$ related to the points of the basis process; cf~\cite[Section~10.2]{blaszczyszyn2017lecture}).~\footnote{ $\P^0$ is defined on the probability space of $\Psi$ with $\Phi=\pi(\Psi)$, see Footnote~\ref{fn.Probability-space}. }

Let 
\begin{equation}\label{eq:xi-intensity-sup}
\g_{\xi, \ms{opt}} := \sup_{\mc L(\Psi) \in \T(\Phi)}\gh(\Psi)\,
 \end{equation}
 denote the \emph{optimal $\xi$-intensity} associated with $\Phi$.
A marking $\Psi$ is \emph{intensity-optimal} if 
$$\gh(\Psi) = \g_{\xi, \ms{opt}}.$$ 
Note that maximizing the $\xi$-intensity for a given base process $\Phi$ (of intensity $\lambda$) is equivalent to maximizing the expected score of its typical point~$\E^0[\xi_0]$.

A priori, it is not clear whether the supremum in~\eqref{eq:xi-intensity-sup} is attained. Confirmation of this property, under appropriate conditions regarding the score function, is our first result presented in Theorem~\ref{existThm} in Section~\ref{s.results}.

\subsubsection{Locally optimal marking}
Optimizing the intensity is only one approach towards optimal stationary markings.
 Alternatively, we work on the level of realizations.
In this regard, for two markings
$\psi, \psi' \in \NMM$ of the same base configuration $\varphi=\pi(\psi)=\pi(\psi')$
define the {\em score difference} 
$$
 \psi' \Delta_{\xi,x} \psi := \xi(\shift_x \psi') - \xi(\shift_x \psi)
$$
evaluated at the point $x\in\varphi$ 
and the aggregated score difference for the entire base configuration
\begin{align}
	\label{scoreDiff}
\psi'\Delta_{\xi} \psi:= \sum_{x \in \varphi} \psi' \Delta_{\xi, x} \psi.
\end{align}
In general, the infinite series 
in~\eqref{scoreDiff} may be not well-defined as the summands could take arbitrary positive or negative values. This issue disappears if 
$\sum_{x \in \varphi} (\psi' \Delta_{\xi, x} \psi)_+ < \infty$,
where $(a)_+=\max(0,a)$. 
We adhere to the convention $(-\infty - (-\infty))_+ = 0$ meaning that there is no penalty in the score difference if $\xi(\psi) = \xi(\psi') = -\infty$, i.e., if $\psi$ and $\psi'$ are both non-admissible. With this notation, a marking $\psi\in\NMM$ is \emph{locally optimal} marking if
$$ \psi' \Delta_{\xi} \psi \le 0
$$
holds for every marking $\psi'\in\NMM$
of the same base configuration (i.e., $\pi(\psi')=\pi(\psi)$) satisfying $\#(\psi' \setminus \psi) < \infty$ (i.e., marking $\psi'$ is different from $\psi$ at only finite number of points of the base configuration)
and $\sum_{x \in \pi(\psi)} (\psi' \Delta_{\xi, x} \psi)_+ < \infty$.
A stationary marking $\Psi$ is \emph{locally optimal} if it satisfies the above condition almost surely.

\bigskip
 %
 %
Having introduced two different approaches towards optimality begs the question as to when they are equivalent? 
Whereas in the thinning context \cite{thinning}, the distribution of grain sizes appearing in the thinning is the decisive factor, in the generalized setting of optimal markings internal and external stabilization~\cite{stab3} of the score function plays an important role. Loosely speaking, external stabilization means that modifying the mark of one point does not change much the scores outside a bounded neighborhood around that point. Conversely, internal stabilization means that changing the markings outside a bounded neighborhood of a point only has a small influence on the score of the considered point. 
Our second main results, formulated in Theorems~\ref{equiv1Thm} and~\ref{equiv2Thm} in Section~\ref{s.results}, provide sufficient stabilization and moment conditions for the equivalence of intensity- and locally optimal markings.

While Theorems \ref{existThm}--\ref{equiv2Thm} build a solid foundation for the existence of intensity- and of locally optimal markings, the question of distributional uniqueness opens up a wide field of research questions. We touch upon it in Section~\ref{s.uniqueness}, after having presented a variety of examples in Section \ref{exSec}.

\subsection{Results}
\label{s.results}
Our first main result guarantees existence of an intensity-optimal marking under tightness and integrability conditions, which allow one to construct it as a suitable subsequential limit.

	%
	%
A family of stationary markings $\{\Psi_i\}_{i \in I} \subset \T(\Phi)$ is \emph{$\M'$-tight} if the corresponding family of distributions 
of $\M'$-marks of the typical point is (uniformly) tight. In other words, for every $\e > 0$ there exists a compact $K' \subset \M'$ such that 
	\begin{align}
		\label{xisTightEq}
		\sup_{i \in I} \P^0_i(M_0' \not \in K') \le \e,
	\end{align}
where $\P^0_i$ corresponds to the Palm probability on the probability space on which $\Psi_i$ is defined and $M_0'$ is the $\M'$-mark of $X_0=o$.
Note, condition~\eqref{xisTightEq} is trivially satisfied when the mark space $\M'$ is compact, as e.g. for the binary marks $\M' = \{0, 1\}$.

Even if no mass of marks is lost when passing to subsequential limits, we still need to impose additional conditions to ensure that also the $\xi$-intensity is well-behaved. In that regard, a family of real, measurable functions $\{\xi_\e\}_{\e > 0}$ defined on $\NMM^0$ is a \emph{continuous approximation} to $\xi$ if
the following three conditions are satisfied:
\begin{enumerate}
	\item For every admissible $\psi \in \NMM^0$, the sequence $\xi_\e(\psi)$ is increasing in $\e$
and $\lim_{\e \to 0}\xi_\e(\psi) = \xi(\psi)$, 
	\item For each $\epsilon>0$, the discontinuity set 
	of the function $\NMM^0\owns\psi\mapsto\xi_\e(\psi)$ (in vague topology on $\NMM^0$) is a zero-set with respect to the distribution of $\Psi$ under $\P^0$, for any $\Psi \in \T(\Phi)$.
\end{enumerate}

Finally, a measurable function $\ol \xi:\, \NM^0 \to [0, \infty)$ is an \emph{integrable majorant} of $\xi$ if the following two conditions are satisfied:
\begin{enumerate}
	\item \label{cond.majorant1} $\E^0[\ol\xi(\Phi)] < \infty$,
	\item \label{cond.majorant2} $|\xi(\psi)| \le \ol \xi(\pi(\psi))=\ol \xi(\varphi)$
	for any admissible marking $\psi$ of  $\P^0$-almost all $\varphi\in\NM^0$.
\end{enumerate}

The following result will be proved in Section~\ref{ss.loc-to-int}.
\begin{theorem}[Existence of intensity-optimal markings]
 \label{existThm}
	Let $\mc L(\Phi)$, $M'$ and $\xi$ be such that:
	\begin{enumerate}
	  		 \item[{\bf ({C})}] { The set of admissible markings is a closed subset of $\NMM$ or the score function $\xi$ is upper semicontinuous on $\NMM^0$ in the vague topology},
		\item[{\bf (A)}] $\xi$ admits a continuous approximation $\{\xi_\e\}_{\e > 0}$ with all $\xi_\e$ having integrable majorants, and 
		\item[{\bf (T)}] for every $\g \in \R$ the family $\{\Psi \in \T(\Phi):\, \gh(\Psi) \ge \g\}$ is $\M'$-tight.
	\end{enumerate}
	Then, there exists at least one intensity-optimal marking.
\end{theorem}

%
%
Next, we introduce stabilization and moment conditions for the equivalence of intensity- and locally optimal markings in the spirit of \cite{weakStab}. Specifically, define $Q_r = Q_r(o) = [-r/2, r/2]^d$ for $r>0$, and for $\e > 0$  consider measurable functions $R^\e:\,\NMM^0 \to [0, \infty)$.
We call $R^\e$ {\em stabilizing radii} of the score function $\xi$ if for $\P^0$-almost every realization $\varphi$ of $\Phi$ and any its marking $\psi\in\NMM^0$ 
the following conditions hold:
\begin{enumerate}
	\item\label{cond.finitness} {\em (finiteness)} $R^\e(\psi) < \infty$,
	\item \label{cond.dependence} {\em (mark dependence)} $R^\e(\psi)$ depends only on $\varphi$ and the mark $m'_0\in\M'$ of $o$.
\suspend{enumerate}
\smallskip
Moreover, a family of functions $R^\e$ satisfying the above conditions~\eqref{cond.finitness} and~\eqref{cond.dependence} is 
\resume{enumerate}	
\item {\em internally stabilizing} if $|\xi(\psi) - \xi(\psi')| \le \e$ for any two markings $\psi, \psi'$ of $\varphi$ with $\psi \cap Q_{R^\e(\psi)} = \psi' \cap Q_{R^\e(\psi)}$,
\item {\em externally stabilizing} if $\sum_{x_i \in \R^d \setminus Q_{R^\e(\psi)}} |\psi' \Delta_{\xi, x_i} \psi| \le \e$ for every two markings $\psi, \psi'$ of $\varphi$ with $\psi \setminus \{o\} = \psi' \setminus \{o\}$,
\suspend{enumerate}
both properties holding for $\P^0$-almost every realization $\varphi$ of $\Phi$.
 
\smallskip\indent
For some examples, such as minimal matching, imposing external stabilization would be too strong. As an alternative, we define a strengthened variant of the internal stabilization. To this regard, a function $R:\, \NMM^0 \to [0, \infty)$, 
 satisfying conditions~\eqref{cond.finitness} and~\eqref{cond.dependence} above with 
 $R=R^\e$
is 
\resume{enumerate}
\item \emph{strongly stabilizing} if $\xi(\psi) = \xi(\psi')$ for any two markings $\psi, \psi'$ of $\varphi$ such that $o \not \in \bigcup_{x \in \psi' \setminus \psi} (x + Q_{R(\shift_x\psi) \vee R(\shift_x\psi')})$, for $\P^0$-almost every realization $\varphi$ of $\Phi$.
\end{enumerate}

By the covariance principle, we extend the definition of the stabilizing radii to all points $(x_i,m_i,m_i')\in\psi$ of $\psi\in\NMM$: $R^\e_i(\psi)=R^\e(\shift_{x_i}(\psi))$ and similarly for $R_i(\psi)$.
For a stationary $\M'$-marking 
$\Psi = \{(X_i, M_i, M_i')\}_{i}$ 
of $\Phi$ we denote 
$$R_i^\e := R^\e(\shift_{X_i}(\Psi))$$
and similarly for~$R_i$.
Note $R^\e=R^\e_0$ and $R=R_0$
are the respective stabilizing radii of the point located at the origin $X_0 = o$ under the Palm probability $\P^0$ of $\Psi$.

The following result, to be proved in Section~\ref{ss.int-to-local}, allows one to deduce the intensity-optimality from the local optimality.
\begin{theorem}[From intensity- to locally optimal]
 \label{equiv1Thm}	For $\mc L(\Phi)$ and $M'$, assume $\xi$
admits  either a strongly stabilizing radius, or a family of internally and externally stabilizing radii. If
$-\infty<\g_{\xi, \ms{opt}}< \infty$, then every intensity-optimal marking is locally optimal.
\end{theorem}

	%
	%
To deduce intensity-optimality from the local optimality, we propose the following two alternative sets of sufficient conditions.
	%
	%
A distinguished mark $m_o' \in \M'$ is a \emph{neutral mark} if it does not exert adverse effects on the scores of other points; more precisely, such that if $\psi = \{(x_i, m_i, m_i')\}_{i}$ is an $\M'$-marking of $\varphi= \{(x_i, m_i)\}_{i}$ and $\psi^*$ is the $\M'$-marking obtained from $\psi$ by replacing some $m_i'$ by $m_o'$, then
	\begin{align}
		\label{neutralEq}
		\xi(\shift_{x_i}\psi^*) \ge 0 \quad \text{ and } \quad \psi^* \Delta_{\xi, x_j}\psi \ge 0 \text{ for every $j \ne i$}.
	\end{align}

	%
	%
	Some score functions do not admit any natural neutral mark or integrable majorant. As an alternative, in the spirit of~\cite{energy}, a measurable function $\ol \xi:\, \NM^0 \to [0, \infty)$,
	is called {\em score difference majorant} if 
  for $\P$ almost any realization $\varphi$ of the basis process the following inequality holds
\begin{equation}\label{e.finite-energy}|\psi \Delta_{\xi, x_i} ((\psi \cap B) \cup (\psi' \setminus B))| \le \ol\xi(\shift_{x_i}\varphi)
\end{equation}
	for any two admissible markings $\psi$, $\psi'$ of $\varphi$, any Borel set $B \subset \R^d$ and any $x_i \in \varphi \cap B$. (A similar property is called {\em finite-energy property} in~\cite{energy}.) In words, the score difference majorant $\ol\xi$ controls how much changing the marks away from a given point can influence the score at that point. In particular, by~\eqref{e.finite-energy} $((\psi \cap B) \cup (\psi' \setminus B)$ is necessarily an admissible marking.

In both scenarios, we need stronger moment assumptions on the integrable $R^\e(\psi)$ and $\ol \xi$.
The following result will be proved in Section~\ref{ss.loc-to-int}. 
\begin{theorem}[From locally to intensity-optimal]
		\label{equiv2Thm}
	Assume $\mc L(\Phi)$, $M'$ and $\xi$  a family of internally and externally stabilizing radii $R^\e$.
Assume
		\begin{enumerate}
			\item $\xi$ takes values in $\{-\infty\}\cup[0,\infty)$ and admits a neutral mark and an integrable majorant $\ol\xi$, or 
			\item $\xi$ takes values in $[-\infty,0]$ and admits some score difference majorant $\ol\xi$,
		\end{enumerate}
		such that $\E^0[\ol\xi^p]<\infty$ and 
$\E^0[(R_0^\e)^{dp/(p-1)}] < \infty$ 	for some $p>1$, every $\e > 0$ and every admissible marking $\Psi \in \Ta(\Phi)$. Then, every locally optimal marking is intensity-optimal.
\end{theorem}
Before proving the results, we present several examples of problems for which our assumptions can be verified. 

\subsection{Examples}
\label{exSec}
 The general framework developed 
 above applies to a variety of problems considered separately in the literature. In this section, we briefly present some of them. Specifically, Sections~\ref{sss.HC} -- \ref{sss.LP} discuss, respectively, the problems of the volume-maximal hard-core thinning, the minimal matching, the traveling salesman, and the classical Lilypond model. Examples of problems presented in Sections~\ref{sss.CA} and~\ref{sss.MAC} are drawn from the communication network literature.

\subsubsection{Maximal-volume hard-core thinnings}
\label{sss.HC}
%
%
Let us first revisit our illustrating example of the Introduction, already studied in~\cite{thinning}. Remember, the goal is to define the hard-core thinning of a germ-grain model maximizing the fraction of the space covered by the retained sets.
%
%
In this example $\M = \mathbb K$ is the space of convex, compact bodies in $\R^d$ with $|M|$ denoting the Lebesgue measure of the body $M\in\M$. A thinning of the germ-grain model $\Phi = \{(X_i,M_i)\}$ with germs $\{X_i\}$ and centered (not necessarily independent) grains $\{M_i\}$
 can be thought of as a binary marking, i.e., $\M' = \{0, 1\}$.
The mark $M' = 1$ means retention of the grain $X_i \oplus M_i = \{X_i + m:\, m \in M_i\}$ and $M' = 0$ its deletion. A given thinning $\Psi$ is a \emph{hard-core thinning} if almost surely there is no overlapping of any two retained grains $(X_i\oplus \ms{int}(M_i)) \cap (X_j \oplus \ms{int}(M_j)) = \es$ for all $i,j$ such that $M'_i = M'_j = 1$, where we write $\ms{int}$ for the topological interior.
As the score function, we take the volume $|M_0|$ of the grain at $o$ provided it is not deleted and not overpalled with another retained grain. That is for $\Psi\in\NMM^0$
$$
\xi(\Psi) = M'_0\times
\begin{cases}
|M_0| &\text{if $(X_i\oplus \ms{int}(M_i)) \cap \ms{int}(M_0) = \es$ for all $i$ such that $M_i'=1$,}\\
	-\infty &\text{otherwise,}
\end{cases}
$$
with $0 \times -\infty =0$. Note that the $\xi$-intensity of $\Psi$ equals the volume fraction of the process of the retained grains, provided it is a hard-core process, and $-\infty$ otherwise. 
Indeed, for a hard-core thinning $\Psi$ we have 
$\P^0$-a.s. $\xi(\Psi)=|M_0|M_0'$, and by the Campbell
formula interpreted as the mass transport
formula between $\Psi$ and the Lebesgue measure, see e.g.~\cite[(10.3.1)]{blaszczyszyn2017lecture},
\begin{align*}\g_\xi(\Psi)& = \lambda \E^0[|M_0|M'_0] 
=\E\Bigl[\sum_{X_i\in\Phi}
\ind(o\in X_i\oplus M_i)M_i'\Bigr]
=\P\Bigl(\,o\in\bigcup_{X_i\in\Phi, M_i'=1}X_i\oplus M_i\,\Bigr),
\end{align*}
where the last equality follows from the hard-core thinning assumption.
Obviously $\g_{\xi, \ms{opt}}\ge 0$.

%
%
 Note first that by taking
 the topological interiors of the grains for the non-overlapping constraint, we ensure that { both the set of admissible markings is closed and
 $\xi$ is upper semicontinuous.}  Moreover, we can choose $\xi_\epsilon(\Psi):= |M_0| M_0'$ for all $\epsilon>0$ as continuous approximations. The function  $\overline \xi(\Phi) = |M_0|$ is an integrable majorant 
 of $\xi_\epsilon(\Psi)$  provided $\E^0[|M_0|] < \infty$. Since $\M'$ is compact, the family $\T(\Phi)$ is $\M'$-tight. 
 Thus, 
 if $\E^0[|M_0|] < \infty$ then 
 by Theorem~\ref{existThm} stationary, volume-maximal, hard-core thinnings exist. Note, this result is non-trivial only if $\g_{\xi, \ms{opt}}>0$.

The following random variable 
\begin{equation}\label{e.R0-hard-core}
R_0 = \inf\{r > 0: (X_i \oplus \ms{int}(M_i)) \cap \ms{int}(M_0) = \es \text{ for every $X_i \not \in Q_r$}\}
\end{equation}
is a strongly 
stabilizing radius provided $\P^0(R_0<\infty) = 1$. 
Under this assumption, by Theorem~\ref{equiv1Thm},
any stationary, volume-maximal, hard-core thinning is locally optimal. Note, this implies also that $\g_{\xi, \ms{opt}}>0$.

In order to establish conditions under which any locally optimal, stationary hard-core thinning is intensity optimal, note that $\xi$ admits $m'_o = 0$
as the neutral mark. Some moment of order $p>1$ of the majorant $\overline \xi(\Phi) = |M_0|$ has to be assumed.

Note $R_0^\e:=R_0$ with $R_0$ defined in~\eqref{e.R0-hard-core} 
is a family of internally and externally stabilizing radii (provided $R_0$ is finite) but this time we need finiteness of its higher order moment $dp/(p-1)$. In this regard, assume that $\Phi$ is a Poisson-Boolean model, i.e., an independently marked Poisson process of germs. Then, writing 
$M_0, M_1$ for two independent copies of the typical (generic) grain and 
$\vM = \{-x: x\in M\}$, 
the Mecke-Slivnyak formula \cite[Theorem 1.4.5]{baccelli2009stochastic2} yields that
\begin{align*}
	\P^0(R_0 > r) & = \P^0\big((X_i \oplus \ms{int}(M_i)) \cap \ms{int}(M_0) \ne \es \text{ for some $X_i \not \in Q_r$}\big)\\
	&\le \E^0\big[\#\{X_i \in \Phi \setminus Q_r:\, o \in (X_i \oplus \ms{int}(M_i)) \oplus \ms{int}(\vM_0) \}\big]\\
	& = \la \Big[\int_{\R^d \setminus Q_r}\P(x \in \ms{int}(M_1) \oplus \ms{int}(\vM_0))\d x \}\Big]\\
	& = \la \E\big[|(M_1 \oplus \vM_0)\setminus Q_r|\big].
\end{align*}
This provides a concrete starting point for further investigations, for instance when the tail behavior of the diameter of the generic grain is available. 

\subsubsection{Minimal matchings}
\label{sss.mm}
This example addresses some question of minimal matchings of point processes studied in \cite{holroyd, match}.
The points of the base process $\Phi$ are marked by the elements of the space $\M = \{\ms b, \ms r\}$ consisting of two different elements. This given marking 
splits the points of the base process into two processes of $\ms b$-``blue'' and $\ms r$-``red'' points. Assuming that the two processes have equal intensities,
one aims at matching blue and red points together while minimizing the distance between the matched partners. 

In order to cast this problem into our optimal marking framework, 
we take $\M' = \R^d$ and for a (red or blue) point $X_i$ in $\Phi$, we interpret its mark $M'_i\in\M'$ as the relative position of its (blue or red, respectively) matching partner in~$\Phi$. 
We consider the following score function, with $|\cdot|$ denoting the Euclidean distance on $\R^d$, and with the value $-\infty$ 
corresponding to a non-admissible marking (not corresponding to a valid matching)
$$
\xi(\Psi) = 
\begin{cases}
	-|M_0'| &\text{if $X_0+M_0'=M_0'\in\Phi$ is a valid matching partner of $X_0=o$,}\\
-\infty &\text{otherwise},
\end{cases}
$$
where the matching validity 
means that $M_0'\in\Phi$,
$M_0'$ is matched to $X_0=o$
(in symbols: $M_0'(\shift_{M'_0}\Psi)=-M_0'$),
and the points $X_0$ and $M_0'$ have opposite colors (marks in $\M$). 

{ Note that 
both the set of admissible markings is closed and
 $\xi$ is upper semicontinuous.} 
 Moreover, we may choose $\xi_\e(\Psi) = -\min\{\e^{-1}, |M_0'|\} $ as continuous approximations { trivially admitting $\overline\xi_\e\equiv \e^{-1}$ as integrable majorants.}
 Note also, that $\xi$ satisfies tightness condition {\bf (T)} in Theorem~\ref{existThm}.
Indeed, for any $\gamma<0$ and $\e>0$ taking $r=-\gamma/(\e\lambda)$ one has
$\P^0_\Psi(|M_0'| > r) \le \e$
for any valid matching $\Psi$ such that 
$\g_\xi(\Psi)=-\lambda\E^0_\Psi[|M'_0|] \ge \gamma$.
 Additionally, $R(\psi) = |M_0'|_\infty$ 
is a strongly stabilizing radius.

Hence, Theorems \ref{existThm} and \ref{equiv1Thm} imply that an intensity-minimal matching exists and any such intensity-minimal matching is also locally minimal provided that $\g_{\xi, \ms{opt}} \ne -\infty$. At this point, we can draw a connection to Question 3 on the existence of locally minimal stationary matchings raised in \cite{holroyd} for the Poisson case. In \cite[Theorems 1, 2]{match}, it is shown that for equal intensities of blue and red points, in dimension $d \ge 3$ the optimal intensity $\g_{\xi, \ms{opt}}$ is indeed finite, so that all intensity-minimal matchings are indeed locally minimal. On the other hand, in dimension 1 and 2 we have $\g_{\xi, \ms{opt}} = -\infty$ and, hence,  all matchings are (trivially) intensity optimal, whereas it is not at all obvious whether there exist locally optimal ones.

We also note that $\xi$ does not exhibit external stabilization. Indeed, when the origin is already matched, then changing the mark of any other point to match with the origin always leads to an non-admissible configuration. 

\subsubsection{Traveling salesman}
\label{sss.ts}
On a finite set of points, the traveling salesman problem asks for a tour of minimal length visiting each of the points at least once. Although this definition does not make sense in a stationary setting, a reasonable extension could consist in considering 2-regular graphs without finite cycles. That means, $\M = \{*\}$ (irrelevant) and $\M' = \R^d \times \R^d$. Hence, any stationary marking can be represented as $\Psi = \{(X_i, \Mi_i, \Mo_i)\}_{i \ge 1}$. Now, we say that $(X_{i_1}, \Mi_{i_1}, \Mo_{i_1}), \dots, (X_{i_k}, \Mi_{i_k}, \Mo_{i_k})$ form a \emph{finite cycle} if for some $k\ge2$, $X_{i_j}+\Mi_{i_j} = X_{i_{j - 1}}$ and $X_{i_j}+\Mo_{i_j} = X_{i_{j + 1}}$ for all $j \in \{1, \dots, k\}$ and distinct $X_{i_j}$,
where the sub-indices are understood to be taken cyclically mod $k$. Note, in particular this means that 
$\Mo_{i_j}=-\Mi_{i_{j+1}}$.

We consider the following upper semicontinuous score function with the value $-\infty$ corresponding to a non-admissible marking (not corresponding to a valid solution of the salesmen problem)
$$
\xi(\Psi) =
\begin{cases}
 - \frac\lambda2(|\Mi_0| + |\Mo_0|) &\text{if $(X_0=o, \Mi_0, \Mo_0)$ belongs to a finite cycle,}\\
-\infty &\text{otherwise}
\end{cases}
$$
With the above definition, the $\xi$-intensity corresponds to the mean path length traversed by the salesmen per unit of volume.

Arguing as in Section \ref{sss.mm}, we see that the conditions of Theorem \ref{existThm} are satisfied, so that we conclude the existence of an intensity-optimal marking. However, in this setting, it is not clear how to define a radius of stabilization. Indeed, changing the marks of distant points could still give rise to a cycle containing the origin, thereby leading to a non-admissible configuration.

\subsubsection{Lilypond model}
\label{sss.LP}
Consider the growth model, where at time 0, every point of a stationary point process starts expanding a spherical grain around it with unit speed in all directions and ceases its growth at the instant when it collides with another grain. This is the intuitive description of the lilypond model from \cite{llp0}.
In particular, it is a hard-core, germ-grain model with spherical grains. Additionally, the variable radii are subject to some maximization under hard-core constraints.
Formally, let $\Phi = \{X_i\}_{i \ge 1}$ be the base process
on $\R^d$, so that the mark space { $\M=\{*\}$ is irrelevant.} The space of optimization marks $\M' = [0, \infty)$
corresponds to the possible radii of the spherical grains.
Denote by $M'_i$ the final radius of the grain of $X_i$. Note that neither existence nor uniqueness of a marking $M'_i$ reflecting the lilypond growth dynamics is evident. In fact, the problem is shown in~\cite{heveling}
to be equivalent to the existence and uniqueness of the intensity-optimal marking with respect to the following upper semicontinuous score function
$$\xi(\Psi) = 
\begin{cases}
M_0' - \inf_{i \ge 1}U(i) &\quad\text{if $M_0'\le\inf_{i \ge 1}U(i)$ for all $i \ge 1$,} \\
-\infty&\quad\text{otherwise,} 
\end{cases}
$$
where 
$$U(i) := \max\Bigl(\frac{|X_i|}2,|X_i| - M'_i\Bigr).$$
Note that the marking $\Psi$ is hard-core --- i.e., satisfies $|X_i - X_j| \ge M_i' + M_j'$ for all $i \ne j$ --- if it is admissible in the sense of Section \ref{defSec}; see \cite[Proposition 3.2]{heveling}.
It is shown in~\cite[Theorem 4.1]{heveling} that there exists a unique intensity-optimal and locally optimal marking
achieving $\g_{\xi,\ms{opt}} = 0$. Although the explicit construction in \cite[Theorem 4.1]{heveling} does not require us to verify the conditions of Theorems \ref{existThm} and \ref{equiv1Thm} any longer, it is still instructive how the general framework would be applied in the present example. Additionally, Theorem \ref{equiv2Thm} yields a characterization of locally optimal markings that is not an immediate consequence of \cite{heveling}.
%
%

First, we use $\xi_\e(\Psi) = M_0' - \inf_{i \ge 1}U(i)$ as a continuous approximation for every $\e > 0$. 
Note that $\xi$ satisfies tightness condition {\bf (T)} in Theorem~\ref{existThm}.
Indeed, for any hard-core marking, the radius of the typical grain of the typical point is not larger than the distance to its nearest neighbor 
$$R_{0, \ms{NN}} = \inf_{i \ne 0} |X_i |.$$ 
The random variable $R_{0, \ms{NN}}$ is not only $\P^0$ almost surely finite for any base process of non-null intensity
(thus guaranteeing condition {\bf (T)}) but also can be used as an integrable majorant $$\overline\xi(\Psi) = R_{0, \ms{NN}}$$
provided $\E^0[R_{0, \ms{NN}}] < \infty$. Similarly to Example \ref{sss.HC}, it is square-integrable if 
$$\E^0\Big[R_{0, \ms{NN}}\sum_{X_i\in Q_1}R_{i, \ms{NN}}\Big] < \infty.$$
Moreover,
$$R_0 = R_0^\e = \inf\big\{r > 0: B_{R_{i, \ms{NN}}}(X_i) \cap B_{R_{0, \ms{NN}}} = \es \text{ for every $X_i \not \in Q_r$}\}\big\}$$
is a radius of stabilization for every $\e > 0$, which is almost surely finite if $\E^0[R_{0, \ms{NN}}^d] < \infty$.
In fact, $R_0$ is strongly stabilizing for $\xi$.

We will present now two examples drawn from the communication network literature.
\subsubsection{Caching}
\label{sss.CA}
A cache is a component that stores data, so that future requests can be served faster. Caching is used in communication networks 
to bring popular multimedia contents 
close to end-users. Optimal content placement,
especially in spatially distributed caches, is a problem that has recently attracted a lot of attention in the engineering literature. 
In what follows, we will address it using the optimal marking framework.

As in Section~\ref{sss.HC}, let $\M$ be the space of convex, compact bodies in
$\R^d$. The marked point process $\Phi \in \NM$ generates the
germ-grain coverage process $\bigcup_{i \ge 1} X_i\oplus M_i$, where we interpret $X_i$ as the locations of caches in the space and
$X_i\oplus M_i$ as the coverage regions of the respective caches
(from where the items stored in memory of the given cache can be
reached).

Consider $K\ge1$ interpreted as the cache memory size.
Let $\M'$ be the set of all $K$-element subsets of $\{1,2,\ldots\}$.
We interpret $\{1,2,\dots\}$ as the index set of
all possible items and $m'\in\M'$ as a given set of $K$
different items, which can be stored in the memory of a given cache.

If $M'_i\in\M'$ is a mark of some point $(X_i, M_i, M_i')\in\Psi$ of a
marked point process $\Psi\in\NMM$, then all contents with indexes in
$M'_i$ are available in the coverage region $X_i\oplus M_i$.

Consider a probability distribution $\{p_k: k\ge1\}$ on $\{1,2,\dots\}$.
 We interpret $p_k$, $k\ge1$, as the popularity of the item with index~$k$ and we define the score function
$$\xi(\Psi) = \sum_{k\ge1}p_k\int_{M_0}\frac{\ind(k\in M'_0)}{n_k(y)}\,\d y,$$
where
$n_k(y) = \sum_{i \ge 1}\ind(y \in X_i\oplus M_i)\ind(k\in M'_i)$
is the number of caches from which the item $k$ is available at $y$. 
Similarly to the hard-core thinning in Section~\ref{sss.HC}, by the Campbell
formula interpreted as the mass transport
formula between $\Psi$ and the Lebesgue measure, see e.g.~\cite[(10.3.1)]{blaszczyszyn2017lecture},
we have 
$$
 \gh(\Psi) 
 =
 \sum_{k\ge1}p_k
	\P\Big(\,o\in\bigcup_{j:k\in M'_j}(X_j\oplus M_j)\,\Big),
 $$
and thus the $\xi$-intensity becomes the probability that a random item, chosen according to the popularity probability $\{p_k\}_{k \ge 1}$, requested at the origin (or at an arbitrary location) can be found in some cache
covering this location. { Let us call it the {\em content hit probability.}}
Mathematically, it is also the $\{p_k\}$-mixture of the volume fractions of the areas where the item $k$ is available. 

%
%
{ Note that all markings in $\NMM$ are admissible hence (trivially) form a closed set. As in the hard-core examples of Section~\ref{sss.HC}, $\xi$ is upper semicontinuous} but continuous approximations are more involved since $\xi(\Psi)$ can potentially be influenced by a distant germ with a large grain { (potentially making the score function jump up, when this point and its grain disappear in the limit)}. Therefore, we put 
	$$\xi_\e(\Psi) = \sum_{k\ge1}p_k\int_{M_0}\frac{\ind(k\in M'_0)}{n^\e_k(y)}\,\d y,$$
where
$n^\e_k(y) = \sum_{i \ge 1}\ind(y\in X_i\oplus M_i)\ind(k\in M'_i) h_\e(X_i)$ for a continuous and compactly supported function $h_\e:\,\R^d \to [0, 1]$ that tends monotonically to 1 as $\e \to 0$. Again, the score function $\overline\xi(\Phi) = |M_0|$ is an integrable majorant provided $\E^0[|M_0|]<\infty$. If the probability distribution $\{p_k\}_{k\ge1}$ is compactly supported, then $\T(\Phi)$ is $\M'$-tight and by Theorem~\ref{existThm}  there exist stationary cachings maximizing the content hit probability.
Local optimality of these cachings and the converse result can be obtained under the same assumptions as in  hard-core examples of Section~\ref{sss.HC}. Indeed, 
the score function is stabilizing with the stabilization radius $R_0$ given by~\eqref{e.R0-hard-core}. Moreover, we may assume that { there exist $K$ content items of zero popularity. That is, $p_{k_i} = 0$ for some $i=1,\ldots,K$. Then, $\{k_i:i=1,\ldots,K\}$ can serve as neutral mark.}

Assuming independently marked Poisson base process $\Phi$, intensity-optimal {\em independent} markings (optimal within the class 
of independent markings) are characterized in \cite{optCache}. Some Gibbs markings in a finite window are considered in \cite{optGibbs}.

\subsubsection{Medium access control}
\label{sss.MAC}
%
%
	In a wireless communication network, it is often undesirable that all transmitters send signals simultaneously. This would introduce substantial interference, thereby leading to a loss of the global quality of service. For this purpose, in practice \emph{medium access control (MAC)} is implemented to determine which users are granted access to the network infrastructure at a given instance of time. A good policy for MAC should strike a balance between high network availability and limited interference. Then, \cite{infocom} describes a MAC maximizing the sum of logarithmic throughputs, which can be cast in the framework of optimal markings. Additionally, it is shown that this optimum can be conveniently approximated by implementations relying only on local information of the network structure.

%
%
	For ease of presentation, we provide a simplified version of the model. We set $\M = \partial B_1$ as the unit sphere on~$\R^d$, so that $\Phi = \{(X_i, Y_i)\}_{i \ge 1}$ is a stationary point process of transmitters $X_i \in \R^d$, together with an associated receiver at $X_i + Y_i$, such that
	$Y_i \in\M$. In Aloha-type MACs, a transmitter $X_i$ accesses the network independently with a probability $P_i\in (0,1]=\M'$ and these probabilities will be our optimization marks. 
	
	Hence, the stationary marked processes $\Psi$ are written as $\Psi = \{(X_i, Y_i, P_i)\}_{i \ge 1}$. 
The score function in this example should capture the probability of successfully transmitting a message from $X_i$ to $Y_i$, despite the interference generated by other transmitters. This probability sensitively depends on the underlying communication protocol and for the situation described in \cite{infocom} (assuming for simplicity null noise power), it simplifies to 
	$$P_i \prod_{j \not= i}\Big(1 - \frac{P_j}{1 + |X_j -(X_i + Y_i)|^\beta}\Big)$$
	for a model parameter $\beta > d$. Motivated by proportional fairness principles, the score function is chosen as the logarithm of the above expression. That is, 
	\begin{align}
		\label{macEq}
	\xi(\Psi) = \log(P_0) + \sum_{i \not= 0} \log\Big(1 - \frac{P_i}{1 + |X_i - Y_0|^\beta}\Big) \in [-\infty,0].
	\end{align}

	 Remarkably, invoking the Mass Transport Principle, \cite{infocom} provides an explicit construction of the optimal probability as a function of $(X_i, Y_i)$ and $\Phi \setminus \{(X_i, Y_i)\}$, as well as argues its uniqueness. We shall generalize this idea in Section~\ref{ss.uniqueness-MTP}. Nevertheless, as in Example \ref{sss.LP}, it is instructive to consider conditions of Theorems \ref{existThm}--\ref{equiv2Thm} for the present example.

	 %
	 %
{ Even if $\xi$ itself is continuous, we need to construct further continuous approximations of it admitting integrable majorants.} In that regard, we take $$\xi_\e(\Psi) := \max\{-1/\e, \log(P_0)\} + \sum_{i \not =0} \log\Big(1 - \frac{P_i}{1 + |X_i - Y_0|^\beta}\Big) h_\e(X_i),$$
	where again $h_\e:\,\R^d \to [0, 1]$ is a compactly supported continuous function tending monotonically to 1 as $\e \to 0$. Here, we put a lower bound on the term $\log(P_0)$ to be able to construct majorants $\ol\xi_\e:=-1/\e +{\ol\xi}$,
with 	
\begin{equation}\ol\xi=\ol \xi( \Phi) := - \sum_{i \not= 0}\log\Big(1 - \frac1{1 + |X_i - Y_0|^\beta}\Big). \label{e.aloha-majorant}
\end{equation}
These majorants are integrable provided $\E^0[\ol \xi(\Phi)]<\infty$
--- a condition which can be verified for Poisson network.

 Note also, the problem satisfies condition {\bf (T)} of Theorem~\ref{existThm}. Indeed, for any marking $\Psi$ we have
 $\xi(\Psi) \le \log(P_0)$
 and consequently 
 for $a\in(0,1)$, $\gamma\in(-\infty,0)$
 \begin{align*}
   \P^0\{\,P_0<a\,\}&=\P^0\{\,\log(P_0)<\log(a)\,\}\\
&\le \frac{\E^0[\log(P_0)]}{\log(a)}\\
&\le\frac{\gamma}{\log(a)}
 \end{align*}
for any marking $\Psi$ such that $\E^0[\xi(\Psi)]\ge\gamma$. This allows us to conclude Theorem \ref{existThm}.

Next, we have internal and external stabilization radii
	$$R_0^{\varepsilon, \ms i} = \inf\Big\{r > 0:\, -\sum_{i \ge 1}\log\Big(1 - \frac1{1 + |X_i - Y_0|^\beta}\Big) \one(X_i \not \in Q_r) \le \varepsilon\Big\}$$
	and
	$$R_0^{\varepsilon, \ms e} = \inf\Big\{r > 0:\, -\sum_{i \ge 1}\log\Big(1 - \frac1{1 + |Y_i|^\beta}\Big) \one(X_i \not \in Q_r) \le \varepsilon\Big\}$$
	respectively.
Thus assumptions of Theorem~\ref{equiv1Thm} are satisfied provided $\E^0[\ol \xi(\Phi)]<\infty$ (implying $\g_{\xi, \ms{opt}}> - \infty$).

Regarding conditions in Theorem \ref{equiv2Thm},
interestingly, in this example there is no natural choice of a neutral mark since $P_i$ appear with different signs in the two summands in \eqref{macEq}. {However, the model admits $\overline\xi$ given by~\eqref{e.aloha-majorant} as a score difference majorant.} Indeed, in order to verify property~\eqref{e.finite-energy} let $\psi$ and $\psi'$ be two admissible markings of a base configuration~$\varphi\in\NMM^0$
	with common $P_0>0$.
	 Then, the $\log(P_0)$
	 contributions cancel, 
	 admissibility implies that the series in~\eqref{macEq} converge,
	 and we obtain that 
	$$|\psi \Delta_{\xi, o} \psi'| \le \overline\xi(\varphi).$$
	
The moment conditions 
are also easy to verify 
under Poisson assumption.
Indeed, combining the exponential Markov inequality with a computation involving Laplace transforms, one can show that the proposed radii of stabilization exhibit exponential tails, thus have finite moments of all orders.

\subsubsection{Unimodular Minkowski dimension}
\label{sss.UMD}
A large part of the terminology that we develop here for stationary markings carries over to the setting of unimodular discrete spaces as defined in \cite{unimDisc} and motivated from unimodular random networks \cite{unimNet}. Loosely speaking, unimodular discrete spaces are random metric spaces, where each point is isolated and such that the notion of a typical node can be defined in such a way that the mass transport principle holds. Similarly to stationary thinnings of point processes, \cite{unimDisc} introduces the concept of equivariant coverings of unimodular discrete spaces in order to extend the notion of the Minkowski dimension. Hence, the optimal coverings appearing in \cite[Definition 3.9]{unimDisc} are very close in spirit to the intensity-optimal markings in the present setting.

\subsection{Uniqueness}
\label{s.uniqueness}
We have seen in Theorems \ref{existThm} -- \ref{equiv2Thm} that intensity-optimal and locally optimal markings exist under rather general conditions. However, the question of distributional uniqueness is wide open. We discuss below different arguments that can be used in this regard, depending on the specific model. 

\subsubsection{Absence of percolation}
In the context of maximal hard-core thinnings (Section~\ref{sss.HC}),  the caching problem (Section~\ref{sss.CA}), and similar ones,
the absence of percolation can be used to prove 
the distributional uniqueness of the optimal marking.
Roughly speaking, in this case, any (globally) optimal marking needs to be also locally optimal within each finite cluster, and thus it is unique provided uniqueness of the local optimal choice. Using this approach, in~~\cite{thinning}, uniqueness of the maximal hard-core thinning has been proved in regimes of sub-critical and barely super-critical percolation
assuming absolute continuity of the factorial moment measures of the germ point process marked by the volumes of the grains.
\subsubsection{Implicit characterizations}
\label{ss.uniqueness-LP}
Consider the setting, where the desired markings are solutions to an equation of the form $H(\Psi) = 0$ for some real measurable function $H$ defined on $\NMM^0$. We can cast this into the setting of optimal markings via the score function
\begin{align*}
\xi(\Psi) = 
\begin{cases}
	H(\Psi) & \text{if $H(\Psi) \le 0$,}\\
-\infty & \text{otherwise.}
\end{cases}
\end{align*}
In particular, the lilypond model of Section~\ref{sss.LP} is of this form. 
 If $\Psi \in \T(\Phi)$ is such that $H(\Psi) = 0$ almost surely under $\P^0$, then $\g_{\xi, \opt} = 0$ and $\Psi$ is both intensity-optimal. Conversely, if $\g_{\xi, \opt} = 0$, then an intensity-optimal marking $\Psi$ also satisfies $H(\Psi) = 0$ almost surely. Consequently,  distributional 
 uniqueness of the solution to $H(\Psi) = 0$ is equivalent to the distributional uniqueness of the intensity-optimal markings. This is the case for the lilypond model on any point process; cf~\cite{heveling}.

\subsubsection{Mass transport principle (unimodularity)}
\label{ss.uniqueness-MTP}
Inspired by the MAC example of Section~\ref{sss.MAC}
consider a score function
in the form
\begin{equation}\label{e.linearxi}
\xi(\Psi) = 
	\sum_{i}
H(X_i, M'_i, \shift_{X_i}\Phi)\,,
\end{equation}
for some real, measurable function $H$ 
defined on $\R^d\times\M'\times\NM^0$
with values either in $[0, \infty) \cup
 \{-\infty\}$ or in $[-\infty, 0]$. Note that the series in~\eqref{e.linearxi} is well defined, with its value taken to be $-\infty$ if at least one term takes this value. 
In this case, finding an intensity-optimal marking reduces to the maximization
of some real valued function over $\M'$.

Indeed, define a function~$f$ on $\M'\times\NM^0$ 
\begin{equation}\label{e.maxM'-MTP}
f(m',\varphi) = \sum_{x\in\varphi} H(-x,m',\varphi).
\end{equation}
\begin{proposition}[Representation of the optimal intensity]
	\label{prop.opt-MTP}
	Assume that $f$ is continuous in $m'$ for $\P^0$-almost all realizations $\varphi$ of the base process $\Phi$. Then,
$$\g_{\xi,\ms{opt}} = \E^0[\sup_{m'\in\M'} f(m',\Phi)]\,.$$
\end{proposition}
\begin{proof}
	The result follows from the Mass Transport Principle (Unimodularity) of stationary point processes allowing one to express the $\xi$-intensity of $\Psi$
in terms of the base process~$\Phi$ and only the typical mark $M'_0$,    see e.g. \cite[Theorem 10.2.7, Remark 10.2.9]{blaszczyszyn2017lecture}
 and \cite[Remark 3.8]{mtp2}. Here,
 $$
	\gamma_\xi(\Psi) = \la\E^0\Bigl[\sum_{i} H(X_i,M'_i,\theta_{X_i}\Phi)\Bigr] = \la\E^0\Bigl[\sum_{i} H(-X_i,M'_0,\Phi)\Bigr] = \la\E^0[f(M'_0,\Phi)]\,.
	$$
Clearly,
$$\sup_{\mc L(\Psi) \in \T(\Phi)}\gh(\Psi) = \sup_{\mc L(\Psi) \in \T(\Phi)} \la \E^0[f(M'_0,\Phi)] \le \la\E^0[\sup_{m' \in \M'} f(m',\Phi)]\,.$$
On the other hand, for a given $\e>0$, consider a measurable function
	$m'_\e:\NM^0\longrightarrow \M'$ such that
$f(m'_\e(\varphi),\varphi)\ge \sup_{m' \in \M'}f(m',\varphi) - \e$. For the marking $\Psi_\e = \{(X_i, M_i, m'_\e(\theta_{X_i}\Phi))\}_{i}$
we have
$$\sup_{\mc L(\Psi) \in \T(\Phi)}\gh(\Psi)\ge
\gh(\Psi_\e)
 =
\la\E^0[f(m'_\e(\Phi),\Phi)]\ge
\la\E^0[\sup_{m' \in \M'} f(m',\Phi)] - \la\e\,,$$
which concludes the proof.
\end{proof}

To simplify measurability questions, we assume that $\M' = [0, \infty)$, but stress that the basic method carries over seamlessly to more general spaces.
The existence and (distributional) uniqueness of the intensity-optimal marking depend on the same properties of $\sup_{m' \ge 0} f(m',\Phi)$.
\begin{proposition}[Existence criterion]
	\label{prop.ex-MTP}
	Assume that $f$ is  continuous in $m'$ for $\P^0$-almost all realizations $\varphi$ of the base process $\Phi$. 
There exists an intensity-optimal marking if and only if $$\arg\max_{m' \ge 0}f(m',\varphi)\not = \es$$ for $\P^0$-almost all realizations $\varphi$ of the base process $\Phi$.
\end{proposition}
\begin{proof}
First, assume that the considered $\arg\max$ is not empty.
Let $\mc M':\,\NM^0\longrightarrow[0, \infty)$ be defined such that $\mc M'(\varphi)$ is the smallest element in $\arg\max_{m' \ge 0}f(m',\varphi)$. Then, the marking $\Psi$ with $M'_i:= \mc M'(\theta_{X_i}\Phi)$
is intensity-optimal. 
	
	On the other hand, consider
an intensity-optimal marking $\Psi$ and let $M'_0$ be
the value of the mark at the typical point under~$\P^0$. Then, by Proposition~\ref{prop.opt-MTP},
$$\E^0[f(M'_0,\Phi)] =
\sup_{\mc L(\Psi) \in \T(\Phi)}\frac1\la \gh(\Psi) =
\E^0[\sup_{m' \ge 0} f(m',\Phi)]\,,$$
and thus 
$f(M'_0,\Phi) = \sup_{m' \ge 0} f(m',\Phi)$ 
holds $\P^0_\Phi$ almost surely.
In particular,
$\arg\max_{m' \ge 0}\allowbreak f(m',\Phi) \ne \es$.
\end{proof}

\begin{proposition}[Uniqueness criterion]
	\label{prop.uniq-MTP}
	Assume that $f$ is continuous in $m'$ for $\P^0$-almost all realizations $\varphi$ of the base process $\Phi$. 
	Assume also that $-\infty<\g_{\xi,\ms{opt}}<\infty$. 
The marking
$M'_i:= \arg\max_{m' \ge 0}f(m',\shift_{X_i}\Phi)$ is the (distributionally) unique intensity-optimal marking if and only if
the function $f(m',\Phi)$ in~\eqref{e.maxM'-MTP} has a unique maximum in $m'$ for $\P^0$-almost all realizations of the base process $\Phi$.
\end{proposition}

\begin{proof}
	Define the function $\mc M':\,\NM^0\longrightarrow[0, \infty)$ and the marking $\Psi = \{(X_i, M_i, M_i')\}_{i}$ as in the proof of Proposition \ref{prop.ex-MTP}.
	
	First, assume that $\arg\max_{m' \ge 0}f(m', \Phi)$ is $\P^0_\Phi$-almost surely unique. By Proposition~\ref{prop.opt-MTP} $M'_i$ is an intensity-optimal marking, and consider another intensity-optimal marking $\tilde\Psi = \{(X_i, M_i, \tilde M_i')\}_{i}$. Then,
	$$0 = \E^0[f(M'_0, \Phi) - f(\tilde M'_0, \Phi)],$$
	so that $\P^0$ almost surely $M'_0 = \tilde M'_0$ and hence both markings are distributionally equal.
	
	To show the converse, define the function $\wt{\mc M}':\,\NM^0 \longrightarrow [0, \infty)$ such that $\wt{\mc M}'(\varphi)$ denotes an element in $\arg\max_{m' \ge 0}f(m',\varphi)$ that is not the largest. If the latter set consists only of a single element, then $\wt{\mc M}'(\varphi)$ takes this value. This induces a marking $\wt\Psi = \{(X_i, M_i, \wt M_i')\}_{i}$ with $\wt M_i' = \wt{\mc M}'(\varphi)$, which is also intensity-optimal. Since we assumed optimal markings to be distributionally unique, we conclude that $\wt M_0' = M_0'$ holds $\P^0_\Phi$-almost surely. In other words, $\arg\max_{m' \ge 0}f(m',\Phi)$ is unique for $\P^0$-almost every realization of the base process $\Phi$. 
\end{proof}

\subsubsection{High fluctuations of scores} 
Finally, we illustrate how the question of uniqueness depends sensitively both on the geometry of the underlying spaces and on how strongly the scores vary. This is reminiscent of Gibbsian particle systems in statistical physics that are typically unique when the temperature is high, but can be non-unique in low-temperature regimes. A similar observation appears already in~\cite[Theorem 6]{thinning}, where hard-core thinnings of a 
Poisson Boolean model with spherical grains of uniform radii on $[0,1]$ are considered, 
maximizing "exponential" volume. More precisely, the score of the retained 
(spherical) grain of radius~$r$ is equal to $e^{ar}$, 
where $a\ge1$ is a model parameter.
The high variability of the random variable $e^{aU}$,
with $U$ uniform on $[0,1]$, implies the  
 distributional uniqueness of locally maximal hard-core thinnings of this model in any (sub- or super-critical) regime, for all sufficiently large $a$.

In what follows we shall present an example where uniqueness of the optimal marking arises, in a more subtle way, as a consequence of the high fluctuations of the scores.

Loosely speaking, consider points located at the position of the integers. At each $i \in \Z$ we may either put a particle of type `$+$' of weight $M^+_i$, a particle of type `$-$' of weight $M^-_i$, or no particle at all. The weights $\{(M^+_i,M^-_i)\}_{i \in \Z}$ are iid (across~$i$ and between~$+$ and~$-$) and the only restriction we impose is that particles of different type are forbidden to be adjacent. More formally, we consider a base point process with points
at $\Z$ marked by the elements of the space $\M = [0,\infty)^2$ capturing the weights of the two types of particles. We set $\M' = \{0,+,-\}$ allowing one to choose either no particle, or a particle of the desired type. For a stationary $\M'$-marking~\footnote{For simplicity we consider the discrete stationarity, i.e., distributional invariance with respect to translations by 
integer shifts. For continuum stationarity the whole process has to be shifted by a uniform random variable on $[0,1]$.}

$\Psi = \{(i, (M^+_i, M^-_i), M'_i)\}_{i \in \Z}$, we define the score function as
$$
\xi_\Z(\Psi) = 
\begin{cases}
 -\infty &\text{if $M'_0 \ne 0$ and $M'_i = -M'_o$ for some $i \in \{\pm1\}$},\\ 
 M^{M'_0}_0 &\text{otherwise,}
\end{cases}
$$
where we adhere to the convention that $M_0^0 = 0$.
This example is reminiscent of the Widom-Rowlinson model from statistical physics that also implements an exclusion rule for particles of different types \cite{WiRo70}. 
 As long as the weights $M_i^+$ and $M_i^+$ have densities, the score fluctuations are sufficiently strong to enforce distributional uniqueness.

\begin{proposition}[Uniqueness on the integers]
 \label{uniqThm}
 If $\mc L(M_0^+) = \mc L(M_0^-)$ admits the density with respect to Lebesgue measure, then there is a unique locally $\xi_\Z$-optimal marking.
\end{proposition}
\begin{proof}
	The idea of proof is to construct shields against long-range dependence by considering blocks of adjacent sites that all have larger '$+$' than '$-$' marks. More precisely, a discrete interval $[i, j] \subset \Z$ is \emph{blocking} if
	\begin{enumerate}
		\item $M_k^+ > M_\ell^-$ for all $k, \ell \in [i, j]$, and
		\item $\sum_{i \le k \le j} M_k^+ > \sum_{i-1 \le k \le j+1} M_k^-$.
	\end{enumerate}
	Since the marks $\{M_k^+, M_k^-\}_{k \in \Z}$ are iid and absolutely continuous, blocking intervals occur with a positive intensity. Now, we assert that if $[i, j]$ is blocking, then there exists $k \in [i, j]$ such that $M_k' = +$ for every locally optimal thinning $\Psi$. Once this assertion is established, the proof of the theorem is complete. Indeed, if $k<\ell$ are such that $M_k' = M_\ell' = +$, then local optimality already dictates the specific values of $M_m'$ for any $m \in [k, \ell]$.

	It remains to prove the assertion. For this purpose, we distinguish between different possible values of $M_{i-1}'$ and $M_{j+1}'$. If neither of them equals $-$, then we conclude from property (1) that $M_m' = +$ for all $m \in [i, j]$. Hence, we may assume for instance that $M_{i-1}' = -$. If additionally $M_{j+1}' = -$, then property (2) implies that there exists $m \in [i, j]$ with $M_m' = +$. We claim that for this choice of $m$, we still have $M_m' = +$, even if $M_{j+1} = +$ or $M_{j+1} = o$. Indeed, by property (1), there cannot exist $i \le i' \le j' \le j$ such that $M_{i'-1}' = M_{j'+1}' = o$ and $M_k' = -$ for all $k \in [i', j']$. Hence, there exists $m' \in [i, j]$ such that 
	$$M_k' = \begin{cases}
		- &\text{ if $i \le k < m'$},\\
		0 &\text{ if $ k = m'$},\\
		+ &\text{ if $m' < k \le j$}.\\
	\end{cases}$$
	In fact, now property (1) gives that $m' = i$, so that $M_m' = +$, as asserted.
\end{proof}

Conversely, we now provide an example where enforcing boundary constraints is highly costly, thereby giving rise to multiple optimal markings. In order to create large boundaries, we leave the amenable world and study optimal markings on the 3-regular tree $\B$ with distinguished element $o \in \B$. Although this setting is not covered by the stationary framework laid out in this section, it still makes perfect sense to talk about locally optimal markings. 

More precisely, we proceed as in Proposition~\ref{uniqThm} and let $\{M^{\pm}_b\}_{b \in \B}$ be iid nonnegative random variables. After choosing a distinguished root $o \in \B$, the score function reads 
$$
\xi_\B(\Psi) = 
\begin{cases}
 -\infty &\text{if $M'_o \ne 0$ and $M'_b = -M'_o$ for some $b \sim o$},\\ 
 M^{M'_o}_o &\text{otherwise},
\end{cases}
$$
where we adhere again to the convention that $M_o^0 = 0$.
If the weights do not fluctuate too wildly, then both markings $M'_b = +$ for all $b \in \B$ and $M'_b = -$ for all $b \in \B$ are locally $\xi_\B$-optimal.

The conceptual reason behind the non-uniqueness of locally optimal markings in this example lies in the non-amenability of $\B$. Loosely speaking, the boundary of a finite number of negative particles in an infinite sea of positive particles is so large that the fluctuation gains cannot compensate for the large losses incurred in the forbidden zone at the boundary.
\begin{proposition}[Non-uniqueness on a tree]
 \label{nUniqThm}
 If $\mc L(M^+_o) = \mc L(M^-_o)$ is concentrated in the interval $[0.9, 1.1]$, then there exist multiple locally $\xi_\B$-optimal markings.
\end{proposition}
\begin{proof}
	For a vertex set $V \subset \B$ let $\partial V$ denote the vertex boundary of $V$, i.e., all vertices outside $V$ that are adjacent to a vertex in $V$. 
	We start from the marking $\Psi^+$ consisting entirely of pluses, i.e., $M_b' = +$ for all $b \in \B$. We claim that $\Psi^+$ is locally optimal. Since by symmetry, then the same is true for $\Psi^-$, establishing the claim will conclude the proof.

	Let $\Psi$ be an arbitrary marking disagreeing from $\Psi^+$ in finitely many points. That is, there exists a finite vertex set $V \subset \B$ such that in $\Psi$ we have
	\begin{enumerate}
		\item $M_b' = -$ for all $b \in V$,
		\item $M_b' = 0$ for all $b \in \partial V$, and
		\item $M_b' = +$ for all $b \in \B \setminus (V \cup \partial V)$. 
	\end{enumerate}
	Now, using induction argument, one can show that the non-amenability of the graph implies that $\# \partial V \ge \# V$ for any $V \subset \B$. Since the marks are confined to the interval $[0.9, 1.1]$, we therefore conclude that 
	$$\Psi \Delta_{\B} \Psi^+ \le 1.1\#V - 0.9(\#V + \#\partial V) \le -0.7\#V < 0.$$
	Hence, $\Psi^+$ is locally optimal, as asserted.
\end{proof}
From the above proofs it should be evident that establishing uniqueness based on score variability is more easy in lower dimensions where the boundaries of finite sets are smaller. Note, our one-dimensional and tree examples are extreme ones in this regard. It remains an open question whether uniqueness holds for the optimal marking in Widom-Rowlinson type models in higher (but finite) dimension.
    \section{Proof of existence of optimal markings}
\label{existSec}
In this section, we prove Theorem~\ref{uniqThm}.
Similarly to the proof of~\cite[Theorem 1]{thinning}, intensity-optimal markings are obtained as abstract, subsequential distributional limits. 
We first state that under the assumption of $\M'$-tightness, the family $\T(\Phi)$ of stationary markings is closed under distributional
limits with respect to the vague topology on~$\NMM$ (cf~\cite[Section~4.1]{randMeas}). We defer the technical proof to the end of this section.

\begin{lemma}[Stationary, $\M'$-tight markings under distributional limits]
	\label{markStabLem}
	Let $\T' \subset \T(\Phi)$ be an $\M'$-tight subset of stationary markings of $\Phi$. Let $\Psi_n\in\T'$ be such that 
 $\Psi_n$ converge in distribution to a random measure $\Psi$ on $\NMM$. Then $\Psi \in \T(\Phi)$,
and the Palm  distributions of $\Psi_n$ converge in distribution to the Palm distribution of $\Psi$. 
\end{lemma}

Lemma~\ref{markStabLem} reduces the existence proof to establishing a continuity property of the $\xi$-intensity. In the setting of \cite{thinning}, the score function itself is already continuous outside a zero set, but now, in the general case, we have to introduce continuous approximations. To avoid unnecessary redundancy, we will be shorter on the parts of the proof that are similar to \cite[Theorem 1]{thinning} and provide more details for the steps differing substantially from the thinning setting.

%
%
\begin{proof}[Proof of Theorem \ref{existThm}]
Let $\{\Psi_n\}_{n\ge1}$ be a sequence in $\T(\Phi)$ such that
	$$\lim_{n \to \infty} \gh(\Psi_n) = \g_{\xi,\opt}(\Phi).$$
	 Assume $\g_{\xi,\opt}(\Phi)>-\infty$, otherwise the statement of Theorem~\ref{existThm} is trivially satisfied with all markings being intensity-optimal.
	Since $\T(\Phi)$ consists of markings of the common point process $\Phi$, we conclude from statement~\eqref{i.lemma} of Lemma~\ref{markStabLem-general} that $\T(\Phi)$ is tight as a set of marked point processes. Hence, after possibly passing to a subsequence, the sequence $\{\Psi_n\}_{n\ge1}$ converges weakly to a stationary marked point process $\Psi$. 
	Due to assumption {\bf (T)} of Proposition~\ref{existThm}, by Lemma~\ref{markStabLem}, $\Psi\in\T(\Phi)$.
	 Since $\gh(\Psi) \le \g_{\xi,\opt}(\Phi)$ holds by the definition of $\g_{\xi, \opt}(\Phi)$, it remains to show $\gh(\Psi) \ge \g_{\xi,\opt}(\Phi)$.

	%
	%
	First, we claim that $\Psi$ is admissible.
	{ This follows immediately by Portmanteau Theorem if the set of admissible markings is closed, which is one option in condition {\bf (C)}. If not, 	by Lemma \ref{markStabLem}, the Palm versions $\Psi_n^0$ of $\Psi_n$ converge in distribution to the Palm version $\Psi^0$ of $\Psi$. By the other option in condition {\bf (C)},} the score function is upper semicontinuous, so that for $\xi$ taking values in $[0, \infty) \cup \{-\infty\}$, the Portmanteau Theorem implies that 
	$$\P^0(\xi(\Psi) \ge 0) \ge \limsup_{n \to \infty}\P^0(\xi(\Psi_n) \ge 0) = 1.$$
	Similarly, for $\xi$ taking values in $[-\infty, 0]$, the Portmanteau Theorem implies that for every $M \ge 0$	$$\P^0(\xi(\Psi) > -\infty) = \lim_{M \to \infty}\P^0(\xi(\Psi) \ge -M) \ge \lim_{M \to \infty}\limsup_{n \to \infty}\P^0(\xi(\Psi_n) \ge -M) = 1,$$
	where the last equality is due to our assumption $\g_{\xi,\opt}(\Phi)>-\infty$,
	so that $\Psi$ is admissible.

	By assumption, the discontinuity set of $\xi_\e$ is a zero-set with respect to the Palm distribution of any process in $\T(\Phi)$, including the limit $\Psi$ of $\Psi_n$,
	hence $\xi_\e(\Psi_n^0)$ converge to $\xi_\e(\Psi^0)$ in distribution. Moreover, the $\{\xi_\e(\Psi_n^0)\}_{n \ge 1}$ are uniformly integrable because there exists an integrable majorant of~$\xi_\e$
	by condition {\bf (A)}. Hence, invoking \cite[Lemma 4.11]{fmp}, we arrive at
	$$\g_{\xi_\e}(\Psi) = \la\E^0[\xi_\e(\Psi)] = \la\lim_{n \to \infty} \E^0[\xi_\e(\Psi_n)] = \lim_{n \to \infty} \g_{\xi_\e}(\Psi_n).$$

	When sending $\e \to 0$, then $\g_{\xi_\e}(\Psi)$ tends to $\g_{\xi}(\Psi)$. This follows from monotone convergence for negative score functions and from dominated convergence for positive ones. Hence,
	$$\gh(\Psi) = \lim_{\e \to 0}\g_{\xi_\e}(\Psi) = \lim_{\e \to 0}\lim_{n \to \infty}\g_{\xi_\e}(\Psi_n) \ge \lim_{n \to \infty}\g_{\xi}(\Psi_n) = \g_{\xi,\opt},$$
	as asserted. 
\end{proof}

%
%
We conclude this section by proving Lemma \ref{markStabLem}.
\begin{proof}[Proof of Lemma \ref{markStabLem}]
By the second statement of Lemma~\ref{markStabLem-general}, $\Psi$ is a stationary marked point process on $\R^d\times\M\times\M'$ with intensity smaller or equal to~$\lambda$. In order to prove that the distribution of the projection $\pi(\Psi)$ of $\Psi$ on $\R^d\times\M$ is equal to the distribution of the base process $\Phi$, it suffices to establish equality of the Laplace functionals (cf~\cite[Theorem 4.2]{randMeas})
\begin{equation}\label{e.Laplace-equality}
\LL(\pi(\Psi))(g)=\LL(\Phi)(g),
\end{equation}
for all non-negative, continuous, boundedly-supported functions~$g$ on~$\NM$.
Here and in what follows $\LL(\Upsilon)(h):=\E[e^{-\int h\,\d\Upsilon}]$ denotes the Laplace transform of the random measure~$\Upsilon$ evaluated on function $h$, both defined
(depending on the context) either on $\R^d\times\M$ or $\R^d\times\M\times\M'$. 

Denote by $1'$ a constant function on $\M'$ equal to~1.
Then, since $\Psi\in\T(\Phi)$,
$$\LL(\Psi_n)(g\otimes1')=\LL(\pi(\Psi_n))(g)=\LL(\Phi)(g).$$ However, we cannot deduce convergence of $\LL(\Psi_n)(g\otimes1')$ to $\LL(\Psi)(g\otimes1')=\LL(\pi(\Psi))(g)$ directly from the convergence of $\Psi_n$ to $\Psi$
 since the tensor product $g\otimes1'$ is not compactly supported on $\NMM$ when $\M'$ is not compact.
For this reason we approximate $1'$ by non-negative, continuous, compactly supported Borel functions $g'_k$ on $\M'$, such that $1_{K'_k}\le g'_k\le 1$ , for some sequence of compact $K'_k\subset\M'$ increasing to~$\M'$.
Now, for any $k\ge 1$
$$\lim_{n\to\infty}\LL(\Psi_n)(g\otimes g'_k)= \LL(\Psi)(g\otimes g'_k).$$
In order to conclude~\eqref{e.Laplace-equality}, it is enough to show that that 
\begin{align}\label{e.convergence1}
\lim_{k\to\infty}\LL(\Psi)(g\otimes g'_k&)= \LL(\Psi)(g\otimes1')\\
\noalign{and}
\lim_{k\to\infty}\sup_{n \ge 1}\big|\LL(\Psi_n)(g\otimes g'_k) - \LL(\Psi_n)(g\otimes1')\big|&= 0.\label{e.convergence2}
\end{align}
For~\eqref{e.convergence1}, we use some simple bounds and Campbell's theorem to observe that 
\begin{align*}
0\le \LL(\Psi)(g\otimes g'_k)- \LL(\Psi)(g\otimes1')&=
\E\Bigl[e^{-\int gg'_k\,\d\Psi}
\Bigl(1-e^{-\int g (1 - g'_k)\,\d\Psi}\Bigr)\Bigr]\\ 
&\le \E\Bigl[\int g (1 - g'_k)\,\d\Psi\Bigr]\\
&\le \lambda C ||g||_\infty \E^0_{\Psi}[1-g'_k(M'_0)]\\
&\le \lambda C ||g||_\infty \P^0_{\Psi}(M'_0\not \in K'_k)]\to_{k\to\infty} 0,
  \end{align*}
where $C$ is the Lebesgue measure of the projection of the support of~$g$ on~$\R^d$,
with the second inequality holding since $\Psi$ is a stationary marked point process on $\R^d\times\M\times\M'$ with intensity $\lambda$.
For~\eqref{e.convergence2}, 
observe similarly to $\Psi$,
\begin{align*}
0\le \LL(\Psi_n)(g\otimes g'_k)- \LL(\Psi_n)(g\otimes1') 
&\le \E\Bigl[\int g (1-g'_k)\,\d\Psi_n\Bigr]\\
&\le \lambda \int \E^0_{\Psi_n}[g(x,M_0)(1-g'_k(M'_0))]\\
&\le \lambda C ||g||_\infty \P^0_{\Psi_n}\{M'_0\not \in K'_k)]\to_{k\to\infty} 0\quad \text{uniformly in $n$},
\end{align*}
and where the last statement is due to $M'$-tightness of $\Psi_n$.

The convergence of the Palm distributions follows from the statement~\eqref{iii.lemma} of Lemma~\ref{markStabLem-general}.

	\end{proof}

    \section{Proof of equivalence of optimality criteria}
\label{equivSec}

The equivalence result \cite[Theorem 3]{thinning} for thinnings lays out the general strategy for the proofs of Theorems \ref{equiv1Thm} and \ref{equiv2Thm}. However, some of the most critical steps in this special case do not carry over easily to long-range dependencies, and it is here that we provide finer arguments based on the stabilization properties introduced in Section \ref{modelSec}. Therefore, we only provide outlines of the steps that are similar to \cite[Theorem 3]{thinning} and go into details when the stabilization properties enter the stage.

\subsection{From intensity-optimality to local optimality}
\label{ss.int-to-local}
We proceed by contraposition, assuming that $\Psi$ is not locally optimal. Then, by definition, we can achieve a higher $\xi$-score by modifying the marks in a finite number of points. The key step is to leverage stationarity in order to find a stationary family of such modifications. The total impact of these modifications on the score of the typical point is evaluated
using the mass transport principle.

To ease notation, in the rest of this section we write $\Delta_x$ instead of $\Delta_{\xi, x}$ and $\psi\Delta_B\psi' := \sum_{x\in \varphi\cap B} \psi \Delta_x \psi'$ 
 for the sum of the score differences associated with markings $\psi, \psi'$ of a base configuration~$\varphi$, and evaluated at points in a Borel set $B\subset\R^d$.

 Loosely speaking, if a marking $\Psi$ is not locally optimal, then it is possible to achieve a net increase of $\xi$-scores by modifying finitely many marks. More precisely, a \emph{valid swap} of $\Psi$ is a marking $\psi$ of $\Phi$ satisfying $\#(\psi \setminus \Psi) < \infty$, $\sum_{i}(\psi \Delta_{X_i} \Psi)_+ < \infty$ and $\psi\Delta_{\R^d}\Psi > 0$. In order to paste together different swaps, we need to know that they do not compromise on each other's ability to increase the $\xi$-scores if they occur at a sufficient distance. 
 To make this precise, we first associate each valid swap $\psi$ with a specific "center" point $c(\psi)$, say the lexicographic minimum of $\psi \setminus \Psi$. Hence, we can think of compatible valid swaps as a supplementary  marking of $\Psi$.

 The arguments under the assumption of strong stabilization on the one hand and under internal and external stabilization on the other hand are fairly similar. Nevertheless, in order to render the presentation more accessible, we discuss both of them separately.

 %
 %
 We begin with the simpler setting, namely this of strong stabilization. To apply the stabilization bounds, we need to change the markings in a swap one after another. Since this may change the stabilization radii, for any marking $\psi$ of $\Phi=\pi(\Psi)$ and $X_i\in\Phi$
 we introduce 
 $$R_i^{\mmax}(\psi) = \max\{R(\shift_{X_i}(\psi)), R(\shift_{X_i}(\Psi))\}.$$
 Next, let $$\Theta(\psi) = \bigcup_{X_i \in \psi \setminus \Psi} Q_{R_i^{\mmax}(\psi)}(X_i)$$
denote the \emph{influence domain} of the valid swap $\psi$ and define a \emph{$b$-valid swap} of $\Psi$ to be a valid swap $\psi$ such that 
 \begin{enumerate}
	 \item both the diameter and the number of elements of $\Phi \cap \Theta(\psi)$ is at most $b$, and
	 \item 
		 $\psi \Delta_{\R^d} \Psi > 0$.
 \end{enumerate}
 Two valid swaps $\psi_1, \psi_2$ of $\Psi$ are \emph{compatible} if
$$\Big(\bigcup_{X_i \in \Theta(\psi_1)} Q_{R_i^{\mmax}(\psi_1)}(X_i)\Big) \cap \Big(\bigcup_{X_i \in \Theta(\psi_2)} Q_{{R_i^{\mmax}(\psi_2)}}(X_i)\Big) = \es.$$ 
Note, we require not only that influence domains are disjoint, but that the same holds true for a ``second iteration'' of the influence domains.

 %
 %
If $\Psi$ is not locally optimal, then the following analog of \cite[Lemma 4]{thinning} shows how to create a stationary configuration of compatible $b$-valid swaps.
 \begin{lemma}
  \label{compSwapLem}
	 If $\Psi$ is not locally optimal, then for some $b<\infty$ there exists 
	 a stationary marked point process $\mc S = \{(Y_j,\shift_{Y_j}\Psi_j)\}$ 
	of positive intensity, such that 
	$\Psi_j\in\NMM$ are pairwise compatible, $b$-valid swaps of $\Psi$ and $Y_j=c(\Psi_j)$ are the centers of these swaps.
 \end{lemma}
 \begin{proof}
	 Since $\Psi$ is not locally optimal, by stationarity, valid swaps occur with positive intensity. Hence, for some $b > 0$ also the process of valid swaps $\psi$ satisfying 1) the constraints on the diameter and number of elements in $\psi \setminus \Psi$ and 2) $\psi \Delta_{\R^d} \Psi > 0$ is of positive intensity. 
	 Fixing such $b$, then also the process of $b$-valid swaps has a positive intensity. Now, we can apply the same Mat\'ern-type construction as in \cite[Lemma 4]{thinning} to extract a family of compatible, $b$-valid swaps of positive intensity.
 \end{proof}

 Equipped with this key auxiliary result, we now conclude the proof of Theorem \ref{equiv1Thm}.
 
%
%
\begin{proof}[Proof of Theorem~\ref{equiv1Thm} under strong stabilization]
	If $\Psi$ is not locally optimal, then Lemma \ref{compSwapLem} provides a stationary marked point process $\mc S = \{(Y_j,\shift_{Y_j}\Psi_j)\}$ of positive intensity, of pairwise
	compatible, $b$-valid swaps.
	Let $ \Psi^*$ denote the marking obtained from $\Psi$ by implementing all swaps in $\mc S$. If $X_i \in \Phi$ is such that $X_i \not \in \Theta(\Psi^*)$, then strong stabilization gives that $\xi(\shift_{X_i}(\Psi^*)) = \xi(\shift_{X_i}(\Psi))$.
On the other hand, if $X_i\in \Theta(\Psi^*)$, then $X_i \in \Theta(\Psi_{j})$ for some $j$. Hence, using pairwise compatibility of $\{\Psi_j$\} (it is here that we use that twice iterated influence domains are disjoint) and, again, strong internal stabilization, we have for fixed $j$
	$$\sum_{X_i \in \Theta(\Psi_j)}\xi(\shift_{X_i}(\Psi^*)) = \sum_{X_i \in \Theta(\Psi_{j})}\xi(\shift_{X_i}(\Psi_{j})) > \sum_{X_i \in \Theta(\Psi_{j})}\xi(\shift_{X_i}(\Psi)),$$
with the last inequality holding because $\Psi_j$ is a valid swap of~$\Psi$.
Taking the expectation (more formally, using mass transport formula as shall be explained in the proof under internal and external stabilization)
and using the fact that $\mc S$ has positive intensity, we conclude that $\gh(\Psi^*) > \gh(\Psi)$, so that $\Psi$ is not intensity-optimal.
\end{proof}

%
%
We now assume internal and external stabilizations. As many arguments are very similar, we only explain in detail the steps that are substantially different. 
For any marking $\psi$ of $\Phi=\pi(\Psi)$, $X_i\in\Phi$, and any $\epsilon>0$
 we put 
 $$R_i^{\mmax, \e}(\psi) = \max\{R^\e(\shift_{X_i}(\psi)), R^\e(\shift_{X_i}(\Psi))\}$$
and let
$$\Theta^\e(\psi) = \bigcup_{X_i \in \psi \setminus \Psi} Q_{R_i^{\mmax, \e}(\psi)}(X_i)$$
 denote the \emph{$\e$-influence domain} of $\psi\setminus\Psi$.
 
For any $\delta, \epsilon>0$, $b,b',r<\infty$, a
\emph{$(\delta, \e, b, b', r)$-valid swap} is a valid swap $\psi$ such that 
 \begin{enumerate}
    \item\label{swap3} $\psi \Delta_{Q_{t}(c(\psi))}\Psi \ge \delta$ for all $t\ge 1/\delta$, where $c(\psi)$ is the center of the swap $\psi$,
	 \item\label{swap1} both the diameter and the number of elements of $\psi \setminus \Psi$ is at most $b$,
	 \item\label{swap2} both the diameter and the number of elements of $\Phi \cap \Theta^\e(\psi)$ is at most $b'$, 
	 \item\label{swap4}
	 $\Theta^\e(\psi)\subset Q_r(c(\psi))$ and $r>1/\delta$.
 \end{enumerate}
 For $r,\e,\e'>0$, two valid swaps $\psi_1, \psi_2$ are 
 $(\e, \e')$-compatible
 if
 $$\Big(\bigcup_{X_i \in \Theta^\e(\psi_1)} Q_{R_i^{\mmax, \e'}(\psi_1)}(X_i)\Big) \cap \Big(\bigcup_{X_i \in \Theta^\e(\psi_2)} Q_{{R_i^{\mmax, \e'}(\psi_2)}}(X_i)\Big) = \es.$$ 
 Then, the proof of Lemma \ref{compSwapLem} extends to the current setting to give the following result.
 \begin{lemma}
  \label{compSwap2Lem}
	 If $\Psi$ is not locally optimal, then  for some $\delta, \e,\e'>0$, $b,b',r<\infty$ with $\delta-2\e b-\e'b'>0$, 
	 there exists a stationary marked point process $\mc S = \{(Y_j,\shift_{Y_j}\Psi_j)\}$ 
	of positive intensity, such that 
	$\Psi_j\in\NMM$ are pairwise $(r, \e, \e')$-compatible, $(\delta, \e, b, b', r)$-valid swaps
	 of $\Psi$ and $Y_j=c(\Psi_i)\in\R^d$ are the centers of these swaps.
 \end{lemma}
\begin{proof}
Since $\Psi$ is not locally optimal, by the stationarity, there exist a stationary marked point process $\mc S' = \{(Y_j,\shift_{Y_j}\Psi_j)\}$ of positive intensity, of valid swaps of $\Psi$. We shall proceed by the following successive refinement of $\mc S'$ 
to obtain $\mc S$ with swaps satisfying the required conditions. 
First, since $\Psi_i$ are valid swaps of $\Psi$, there exist
$\delta>0$ such the subset of $\mc S'$ satisfying
condition~\eqref{swap3} of $(\delta, \e, b, b', r)$-valid swap has positive intensity. We retain this subset of $\mc S'$. Next, there exist $b<\infty$ such that the retained subset of $\mc S'$ contains further subset of positive intensity satisfying condition~\eqref{swap1} of $(\delta, \e, b, b', r)$-valid swap. We retain this subset. At this stage we choose 
$\e>0$ such that $\delta-2\e b>0$.
Subsequently, for appropriate $b'<\infty$, within the current subset of swaps we can find a   subset of positive intensity satisfying condition~\eqref{swap2} of $(\delta, \e, b, b', r)$-valid swap and choose $\e'>0$ such that $\delta-2\e b-\e' b'>0$. Finally, let us take $r$ large enough to satisfy 
condition~\eqref{swap4} of $(\delta, \e, b, b', r)$-valid swap.
It remains to 
 apply the same Mat\'ern-type construction as in \cite[Lemma 4]{thinning} to extract from the obtained stationary family of 
 $(\delta, \e, b, b', r)$-valid swap the final subset $\mc S$  of pairwise $(r, \e, \e')$-compatible swaps.
\end{proof} 
 
 \begin{proof}[Proof of Theorem~\ref{equiv1Thm} under internal and external stabilization]
 Assume $\Psi$ intensity optimal with finite $\xi$-intensity but not locally optimal and let $\Psi^*$ be the marking of $\pi(\Psi)$ resulting from the implementation of all swaps $\Psi_j$ from a stationary process $\mc S =\{(Y_j,\shift_{Y_j}\Psi_j)\}$ 
 of valid swaps satisfying the assumptions of Lemma~\ref{compSwap2Lem}.
 
 Consider the difference of the expected scores of the typical point $X_0=o$ of $\Phi$ for the two markings: $\E^0[\xi(\Psi^*)]-\E^0[\xi(\Psi)]$,
 where $\E^0$ corresponds to the Palm
 probability given points of $\Phi$, on some suitable probability space on which $\Psi$ and $\mc S$ (and consequently $\Psi^*$) are defined and are jointly stationary. 
Since by the assumption $-\infty<\E^0[\xi(\Psi)]<\infty$
and $\E^0[\xi(\Psi^*)]\le \E^0[\xi(\Psi)]$ we have 
$$\E^0[\xi(\Psi^*)]-\E^0[\xi(\Psi)]=
\E^0[\xi(\Psi^*)-\xi(\Psi)]$$
and we will show that this latter expectation is strictly positive, thus leading to the contradiction.

The main difficulty lies in the fact that we have to evaluate the 
total impact of the implementation of all swaps $\Psi_j\in\mc S$ on one (typical) point of $\Phi$ using available bounds on the total impact of one (typical) swap 
of $\mc S$ on all points $x\in\Phi$. 
Since the number of implemented swaps is infinite, we cannot use a simple union bound, however in the stationary setting one can use the mass transport principle between points of $\Phi$ and $\mc S$ (which is a subset of $\Phi$).
 
In this regard, we shall need to enumerate points $\{X_i\}_{i \in \Z}$ of $\Phi$ in a covariant way. By this we mean that
the order (not the index) of points is invariant with respect to translation by any vector $y\in\R^d$. Such an ordering always exists, in case of dimension $d>1$ perhaps at the price of an extra i.i.d.~marking of $\Phi$ by uniform random variables, see~\cite{timar2004tree}.
Without loss of generality, the point $X_0$ is the one which is closest to the origin~$o$ of $\R^d$ according to some total order on $\R^d$. A given covariant enumeration of points of $\Phi$ induces
the corresponding covariant ordering 
of the swap centers $Y_j=c(\Psi_j)$.

Denote by $\overline\Psi_j$ the marking of $\Phi:=\pi(\Psi)$ resulting from the implementation of all swaps $\Psi_k$ of $\Psi$ with $k\le j$ (this is an infinite subset of swaps).

For any point index $i\in\Z$ ($X_i\in\Phi$) and a swap index $j\in\Z$ (relative to $X_i$, i.e., considered on the configuration $\shift_{X_i}\Phi$
) denote 
\begin{align*}
A(i,j)&:=\xi(\shift_{X_i}\Psi_j)-\xi(\shift_{X_i}\Psi),\\
B(i,j)&:=\xi(\shift_{X_i}\Psi^*)-\xi(\shift_{X_i}\Psi_j),\\
C(i,j)&:=\xi(\shift_{X_i}\overline\Psi_j)-\xi(\shift_{X_i}\overline\Psi_{j-1}).
\end{align*}
Note all terms in the right-hands side of the above expressions are finite, possibly except $\xi(\shift_{X_i}\Psi^*)$.
We have for any $X_i\in\Phi$, $j\in\Z$
\begin{align*}
A(i,j)+B(i,j)&=\xi(\shift_{X_i}\Psi^*) -
\xi(\shift_{X_i}\Psi)\\
\noalign{and}
\sum_{j \in \Z} C(i,j)&=
\lim_{j\to\infty}\xi(\shift_{X_i}\overline\Psi_j) -
\lim_{j\to-\infty}\xi(\shift_{X_i}\overline\Psi_j)
= \xi(\shift_{X_i}\Psi^*) -
\xi(\shift_{X_i}\Psi),
\end{align*}
where the last equality (continuity
of $\xi$ on $\overline\Psi_{j}\mapsto_{j\to\infty}\Psi^*$ and $\overline\Psi_j\mapsto_{j\to-\infty}\Psi$) follows from the internal stabilization of $\xi$.
Moreover,
\begin{align*}
\E^0[\xi(\Psi^*)-\xi(\Psi)]
&= \E^0\Bigl[\sum_{Y_j\in\mc S}
\ind(o\in Q_r(Y_j)) A(0,j)\\
&\hspace{6em}+\ind(o\in \Theta^\e(\Psi_j))B(0, j)\\
&\hspace{6em}-\ind(o\in Q_r(Y_j)\setminus \Theta^\e(\Psi_j)) A(0,j)\\\
&\hspace{6em}+
\ind(o\in\R^d\setminus\Theta^\e(\Psi^*) ) C(0,j)\Bigr].
\end{align*}
Denote by $\lambda_{\mc S}$ the intensity of the process of swaps
and $\E^0_{\mc S}$ the Palm expectation with respect to the process of swap centers. Recall $\lambda_{\mc S}>0$.
By the mass transport formula 
between $\Phi$ and $\mc S$ (cf.~\cite[(10.2.5)]{blaszczyszyn2017lecture}),
\begin{align*}
\frac{\lambda}{\lambda_{\mc S}} \E^0[\xi(\Psi^*)-\xi(\Psi)]
&=\E^0_{\mc S} \Bigl[\sum_{X_i\in\Phi}
\ind(X_i\in Q_r) A(i,0)\\
&\hspace{6em}+\ind(X_i\in \Theta^\e(\Psi_0))B(i, 0)\\
&\hspace{6em}-\ind(X_i\in Q_r\setminus \Theta^\e(\Psi_0)) A(i,0)\\\
&\hspace{6em}+
\ind(X_i\in\R^d\setminus\Theta^\e(\Psi^*) ) C(i,0)\Bigr].
\end{align*}
Observe
\begin{align*}
\sum_{X_i\in\Phi}
\ind(X_i\in Q_r) A(i,0)&\ge \delta\phantom{-b'}\quad \text{by condition~\eqref{swap3} of $(\delta, \e, \e', b, b', r)$-valid swap,}\\
\sum_{X_i\in\Phi}\ind(X_i\in \Theta^\e(\Psi_0))B(i, 0)&\ge -\e' b'\quad\text{by $\e'$-internal stabilization},\\ 
\sum_{X_i\in\Phi}\ind(X_i\in Q_r\setminus \Theta^\e(\Psi_0)) A(i,0)&\ge -\e b\quad\text{by $\e$-external stabilization},\\
\sum_{X_i\in\Phi}\ind(X_i\in\R^d\setminus\Theta^\e(\Psi^*) ) C(i,0)&\ge -\e b\quad\text{by $\e$-external stabilization}.
\end{align*}
Consequently,
$$\lambda \E^0[\xi(\Psi^*)-\xi(\Psi)]\ge
\lambda_{\mc S}( \delta -2\e b-\e'b')>0,$$
which completes the proof of Proposition~\ref{equiv1Thm}
in the case of internally and externally stabilizing score function.
 \end{proof}

\subsection{From local optimality to intensity optimality}
\label{ss.loc-to-int}
To pass from intensity-optimality to local optimality, we take the strategy of \cite[Theorem 3]{thinning}: We take  $\Psi_{\ms i}$ and $\Psi_{\ms l}$ being, respectively, intensity- and locally optimal markings (coupled on some probability space as stationary markings of the common base process~$\Phi$). Replacing a finite number of marks of $\Psi_{\ms l}$ by the respective marks of $\Psi_{\ms i}$ and leveraging the local optimality of $\Psi_{\ms l}$, we shall show that 
$\Psi_{\ms l}$ and $\Psi_{\ms i}$ must have the same intensity.
Again, possible long-range dependencies require more careful arguments involving stabilization radii.
In this regard, let us denote by 
$$\Xis(B) = \{(X_i, M_i) \in \Phi:\, Q_{\max\{R^\e(\shift_{X_i}\Psi_{\ms i}), R^\e(\shift_{X_i}\Psi_{\ms l})\}}(X_i) \subset B \}$$
the set of all points exerting influence at most $\e$ outside a Borel set $B \subset \R^d$.
Also, to lighten notation in the proof, we drop the reference to $\xi$ in $\Delta_{\xi, B}$.
%
%
\begin{proof}[Proof of Theorem~\ref{equiv2Thm}, neutral marks]
For given (small) $\e>0$ and (large) $r<\infty$ define a new marking $\Psi^{\e, r}$ by taking the marks of $\Psi_{\ms i}$ in $\Xis(Q_r)$, the marks of $\Psi_{\ms l}$ in $\Xis(Q_r^c)$, and neutral marks otherwise. Note, using neutral marks guarantees that $\Psi^{\e, r}$ is admissible. We will need the following properties of $\Psi^{\e, r}$, which we prove later.
\begin{align}
\Phi\bigl(\R^d\setminus\Xi^\e(Q_r^c)\bigr)&<\infty,\label{e.finite-mod}\\
\sum_{X_i\in\Phi}(\Psi^{\e, r}\Delta_{X_i}\Psi_{\ms l})_+&<\infty,
\label{e.finite-D}\\
 \E[ \Psi^{\e, r}\Delta_{Q_r \setminus \Xis(Q_r)} \Psi_{\ms i}]&=O(r^{d-1}) \quad\text{as $r\to\infty$},
 \label{e.DQ-surface-order}\\
 \E[\Psi^{\e, r}\Delta_{Q_r^c \setminus \Xis(Q_r^c)} \Psi_{\ms l}]&=O(r^{d-1}) \quad\text{as $r\to\infty$}.
 \label{e.DQc-surface-order}
\end{align}
By~\eqref{e.finite-mod}, \eqref{e.finite-D},
and local optimality of~$\Psi_{\ms l}$ we have $\Psi^{\e, r} \Delta_{\R^d} \Psi_{\ms l} \le 0$. Consequently
	\begin{align*}
		\E[\Psi_{\ms l} \Delta_{Q_r} \Psi_{\ms i}] &= \E[\Psi_{\ms l} \Delta_{\R^d} \Psi^{\e, r}]\\
		&\phantom= +\E[\Psi^{\e, r} \Delta_{\Xis(Q_r)} \Psi_{\ms i}] + \E[ \Psi^{\e, r}\Delta_{Q_r \setminus \Xis(Q_r)} \Psi_{\ms i}] \\
		&\phantom= + \E[\Psi^{\e, r}\Delta_{Q_r^c \setminus \Xis(Q_r^c)} \Psi_{\ms l}] + \E[\Psi^{\e, r} \Delta_{\Xis(Q_r^c)} \Psi_{\ms l}]
				\\
	&\ge \E[\Psi^{\e, r} \Delta_{\Xis(Q_r)} \Psi_{\ms i}] + \E[ \Psi^{\e, r}\Delta_{Q_r \setminus \Xis(Q_r)} \Psi_{\ms i}] \\
		&\phantom= + \E[\Psi^{\e, r}\Delta_{Q_r^c \setminus \Xis(Q_r^c)} \Psi_{\ms l}] + \E[\Psi^{\e, r} \Delta_{\Xis(Q_r^c)} \Psi_{\ms l}],
	\end{align*}
where the first equality is justified by 
the subsequent analysis of the terms.

We begin by analyzing $\E[\Psi^{\e, r} \Delta_{\Xis(Q_r)} \Psi_{\ms i}]$.
 If instead of putting neutral marks in ${Q_r \setminus \Xis(Q_r)}$ we had used the marks from $\Psi_{\ms i}$, then internal stabilization would imply immediately that the absolute value of this term was bounded by $\e \lambda r^d$. Now, the defining property \eqref{neutralEq} of the neutral marks shows that 
$$
\E[\Psi^{\e, r} \Delta_{\Xis(Q_r)} \Psi_{\ms i}]\ge -\e \lambda r^d.
$$
 Using a similar argument for the last term
 and leveraging external stabilization implies $$
 \E[\Psi^{\e, r} \Delta_{\Xis(Q_r^c)} \Psi_{\ms l}]\ge -\e\lambda r^d.
 $$
 Asymptotic properties~\eqref{e.DQ-surface-order} and~\eqref{e.DQc-surface-order} of the remaining two terms imply that 
 these terms are also bounded from below by $-\e \lambda r^d$ for $r$ large enough. Collecting bounds of all terms 
$$\frac1{\lambda r^d} \E[\Psi_{\ms l} \Delta_{Q_r} \Psi_{\ms i}]\ge -4\e$$
for $r$ large enough, thus showing that 
$\Psi_{\ms l}$ has also optimal intensity.
 
In order to complete the proof, we show properties~\eqref{e.finite-mod}--\eqref{e.DQc-surface-order} using appropriate moment assumptions on the score function and the stabilizing radius. Specifically, let 
$$\bar R^\e_i:=\max\{R^\e(\shift_{X_i}\Psi_{\ms i}), R^\e(\shift_{X_i}\Psi_{\ms l})\}.$$
Then, \eqref{e.finite-mod} follows from the finiteness of the expectation
\begin{align*}
\E\Bigl[\#\{X_i\in\Phi:\, X_i\in Q_{r+\bar R^\e_i})\Bigr]&=
\lambda \E^0[(r+\bar R^\e_0)^d]<\infty,
\end{align*}
where we have used the Campbell's formula and
the assumption that $R^\e_0(\Psi_{\ms i})$ and $R^\e_0(\Psi_{\ms l})$ (and consequently 
$(r+\bar R^\e_0)$ have finite $d$th moment under $\E^0$. 

For~\eqref{e.finite-D}, we denote by $\Xi_r := \bigcup_{X_i \in \R^d \setminus \Xis(Q_r^c)} Q_{\bar R_i^\e}(X_i)$ the domain of influence of the points outside $\Xis(Q_r^c)$ and write
\begin{align}\label{e.finite-mod-Xir}
\sum_{X_i\in\Phi}(\Psi^{\e, r}\Delta_{X_i}\Psi_{\ms l})_+
=\sum_{X_i\in\Phi\cap \Xi^c_r}(\Psi^{\e, r}\Delta_{X_i}\Psi_{\ms l})_+
+ \sum_{X_i\in\Phi\cap\Xi_r}(\Psi^{\e, r}\Delta_{X_i}\Psi_{\ms l})_+.
\end{align}
By the external stabilization and (already proved)~\eqref{e.finite-mod} we can bound the first term in~\eqref{e.finite-mod-Xir} by
	$$\sum_{X_i \in \Xi_r^c}(\Psi^{\e, r} \Delta_{X_i} \Psi_{\ms l})_+ \le \e \Phi( \R^d\setminus \Xis(Q_r^c))<\infty.$$
The second term in~\eqref{e.finite-mod-Xir} is finite again by~\eqref{e.finite-mod-Xir}, 
the finiteness of the stabilizing radii $R_i^\e(X_i)<\infty$, and the properties of the score function: $\xi(\shift_{X_i}\Psi_{\ms l})>-\infty$ (since $\Psi_{\ms l}$ is admissible), and $\xi(\shift_{X_i}\Psi^{\e,r})<\infty$
(true for any marking).

Properties~\eqref{e.DQ-surface-order} and~\eqref{e.finite-mod} say that  the 
expected differences between the scores 
of $\Psi^{\e, r}$ and $ \Psi_{\ms i}$,
as well as $\Psi^{\e, r}$ and $ \Psi_{\ms l}$,
aggregated over $Q_r \setminus \Xis(Q_r)$ and $Q_r^c \setminus \Xis(Q_r^c)$, respectively, scale as the $(d-1)$-dimensional volume of the boundary of $Q_r$, i.e. $r^{d-1}$. Indeed, using the
 majorant $\overline \xi_i:=\overline\xi(\shift_{X_i}\Phi)$ to bound the scores of $\Psi_{\ms l}$ and $\Psi^{\e, r}$ (both are admissible)
we have for~\eqref{e.DQc-surface-order}
\begin{align*} |\E[\Psi^{\e, r}\Delta_{Q_r^c \setminus \Xis(Q_r^c)} \Psi_{\ms l}]|&\le
\E\Bigl[\sum_{X_i\in\Phi\cap(Q_r^c \setminus \Xis(Q_r^c))}
|\Psi^{\e, r}\Delta_{X_i} \Psi_{\ms l}|\Bigr]\\
&\le 2\E\Bigl[\sum_{X_i\in\Phi\cap(Q_r^c \setminus \Xis(Q_r^c))}
\overline \xi_i\Bigr]\\
&\le 2\E\Bigl[\sum_{X_i\in\Phi \cap Q_{r+\bar R^\e_i}}
\bar\xi_i\Bigr]\\
&=
2\lambda \E^0[\bar\xi_0(r+\bar R^\e_0)^d]<\infty,
\end{align*}
where the final bound holds by 
the H\"older's inequality 
provided
$\E^0[\bar\xi_0^{p}]<\infty$ and 
$\E[(\bar R^\e_0)^{dp(p-1)}]<\infty$ for some $p>1$.
Similar arguments lead to~\eqref{e.DQ-surface-order}.
This completes the proof of Theorem~\ref{equiv2Thm} with neutral marks.
\end{proof}

\begin{proof}[Proof of Theorem~\ref{equiv2Thm} score difference majorant] This time, we assume negative score functions and construct a new marking $\Psi^r$ by taking the marks of $\Psi_{\ms i}$ in $Q_r$, and the marks of $\Psi_{\ms l}$ in $Q_r^c$. 
Note $\Psi^r$ is again an admissible marking, this time by the property~\eqref{e.finite-energy} and we 
shall be able  
to prove the following counterparts of~\eqref{e.finite-D}--\eqref{e.DQc-surface-order-energy}
relying this time on the score difference majorant~$\ol\xi_i$.
\begin{align}
\sum_{X_i\in\Phi}(\Psi^r \Delta_{X_i}\Psi_{\ms l})_+&<\infty,
\label{e.finite-D-energy}\\
 \E[ \Psi^{r}\Delta_{Q_r \setminus \Xis(Q_r)} \Psi_{\ms i}]&=O(r^{d-1}) \quad\text{as $r\to\infty$},
 \label{e.DQ-surface-order-energy}\\
 \E[\Psi^{r}\Delta_{Q_r^c \setminus \Xis(Q_r^c)} \Psi_{\ms l}]&=O(r^{d-1}) \quad\text{as $r\to\infty$}.
 \label{e.DQc-surface-order-energy}
\end{align}
	Assuming that the above properties hold we have 
 $\Psi^{\e, r} \Delta_{\R^d} \Psi_{\ms l} \le 0$ and consequently
	\begin{align*}
		\E[\Psi_{\ms l} \Delta_{Q_r} \Psi_{\ms i}] 
	&\ge \E[\Psi^{r} \Delta_{\Xis(Q_r)} \Psi_{\ms i}] + \E[ \Psi^{r}\Delta_{Q_r \setminus \Xis(Q_r)} \Psi_{\ms i}] \\
		&\phantom= + \E[\Psi^{r}\Delta_{Q_r^c \setminus \Xis(Q_r^c)} \Psi_{\ms l}] + \E[\Psi^{r} \Delta_{\Xis(Q_r^c)} \Psi_{\ms l}].
	\end{align*}	This time, internal and external stabilization imply directly that the first and last summand are at least $-\e \lambda r^d$. Using the boundary-order scaling of the remaining two summands we show, as before, that $\Psi_{\ms l}$ is necessarily intensity-optimal.
	
It remains to justify properties~\eqref{e.finite-D-energy}--\eqref{e.DQc-surface-order-energy}.
For~\eqref{e.finite-D-energy}
exactly the same arguments can be used as 
for~\eqref{e.finite-D}. (However it is more natural in this case to consider $\Xi_r:= \bigcup_{X_i \in Q_r} Q_{\bar R_i^\e}(X_i)$.)
The main modification of the argument appears in the 
proof of~\eqref{e.DQ-surface-order-energy} and~\eqref{e.DQc-surface-order-energy}, where the first crucial step is done by using score difference majorant satisfying~\eqref{e.finite-energy} (instead of the majorant of the scores)
$$\E\Bigl[\sum_{X_i\in\Phi\cap(Q_r^c \setminus \Xis(Q_r^c))}
|\Psi^{r}\Delta_{X_i} \Psi_{\ms l}|\Bigr]\\
\le \E\Bigl[\sum_{X_i\in\Phi\cap(Q_r^c \setminus \Xis(Q_r^c))}
\overline \xi_i\Bigr].$$
and similarly for the other property.
This concludes the proof of Theorem~\ref{equiv2Thm}.
\end{proof}

	\section*{Acknowledgements}
	The work of CH was funded by The Danish Council for Independent Research—Natural Sciences under Grant DFF -- 7014‐00074 (Statistics for point processes in space and beyond) and by the Centre for Stochastic Geometry and Advanced Bioimaging (funded by the Villum Foundation under Grant 8721). The work of BB was partially supported  by IFCAM project "Geometric statistics of stationary point processes".	We are grateful to G\"unter Last for suggesting us to consider the matching example and to Yogeshwaran D. for his comments on the stabilization techniques.  

\bibliography{mark}

\begin{thebibliography}{10}

\bibitem{unimNet}
D.~Aldous and R.~Lyons.
\newblock Processes on unimodular random networks.
\newblock {\em Electron. J. Probab.}, 12:no. 54, 1454--1508, 2007.

\bibitem{baccelli2009stochastic2}
F.~Baccelli and B.~B{\l}aszczyszyn.
\newblock {\em Stochastic Geometry and Wireless Networks}.
\newblock Now Publishers Inc, 2009.

\bibitem{infocom}
F.~Baccelli, B.~B{\l}aszczyszyn, and C.~Singh.
\newblock Analysis of a proportionally fair and locally adaptive {S}patial
  {A}loha in {P}oisson {N}etworks.
\newblock In {\em INFOCOM, 2014 Proceedings IEEE}, pages 2544--2552. IEEE,
  2014.

\bibitem{unimDisc}
F.~Baccelli, M.-O. Haji-Mirsadeghi, and A.~Khezeli.
\newblock On the dimension of unimodular discrete spaces, part {I}:
  {D}efinitions and basic properties.
\newblock {\em arXiv preprint arXiv:1807.02980}, 2018.

\bibitem{blaszczyszyn2017lecture}
B.~B{\l}aszczyszyn.
\newblock {\em Lecture Notes on Random Geometric Models -- Random Graphs, Point
  Processes and Stochastic Geometry}.
\newblock 2017.
\newblock available at {http://hal.inria.fr/cel-01654766}.

\bibitem{optCache}
B.~B{\l}aszczyszyn and A.~Giovanidis.
\newblock Optimal geographic caching in cellular networks.
\newblock In {\em Communications (ICC), 2015 IEEE International Conference on},
  pages 3358--3363. IEEE, 2015.

\bibitem{stab4}
B.~B{\l}aszczyszyn, D.~Yogeshwaran, and J.~E. Yukich.
\newblock Limit theory for geometric statistics of clustering point processes.
\newblock {\em Ann. Probab.}, 47(2):835--895, 2019.

\bibitem{optGibbs}
A.~Chattopadhyay, B.~B{\l}aszczyszyn, and P.~Keeler.
\newblock Gibbsian on-line distributed content caching strategy for cellular
  networks.
\newblock {\em IEEE Transactions on Wireless Communications}, 17(2):969--981,
  2018.

\bibitem{pp1}
D.~J. Daley and D.~D. Vere-Jones.
\newblock {\em An Introduction to the Theory of Point Processes {I/II}}.
\newblock Springer, New York, 2005/2008.

\bibitem{energy}
A.~Gandolfi, M.~S. Keane, and C.~M. Newman.
\newblock Uniqueness of the infinite component in a random graph with
  applications to percolation and spin glasses.
\newblock {\em Probab. Theory Related Fields}, 92(4):511--527, 1992.

\bibitem{llp0}
O.~H\"aggstr\"om and R.~Meester.
\newblock Nearest neighbor and hard sphere models in continuum percolation.
\newblock {\em Random Structures Algorithms}, 9(3):295--315, 1996.

\bibitem{heveling}
M.~Heveling and G.~Last.
\newblock Existence, uniqueness, and algorithmic computation of general
  lilypond systems.
\newblock {\em Random Structures Algorithms}, 29(3):338--350, 2006.

\bibitem{thinning}
C.~Hirsch and G.~Last.
\newblock On maximal hard-core thinnings of stationary particle processes.
\newblock {\em J. Stat. Phys.}, 170(3):554--583, 2018.

\bibitem{holroyd}
A.~E. Holroyd.
\newblock Geometric properties of {P}oisson matchings.
\newblock {\em Probab. Theory Related Fields}, 150(3-4):511--527, 2011.

\bibitem{match}
A.~E. Holroyd, R.~Pemantle, Y.~Peres, and O.~Schramm.
\newblock Poisson matching.
\newblock {\em Ann. Inst. Henri Poincar\'{e} Probab. Stat.}, 45(1):266--287,
  2009.

\bibitem{randMeas}
O.~Kallenberg.
\newblock {\em Random Measures}.
\newblock Akademie-Verlag, Berlin, 1983.

\bibitem{fmp}
O.~Kallenberg.
\newblock {\em Foundations of Modern Probability}.
\newblock Springer-Verlag, New York, second edition, 2002.

\bibitem{mtp2}
G.~Last and H.~Thorisson.
\newblock Invariant transports of stationary random measures and
  mass-stationarity.
\newblock {\em Ann. Probab.}, 37(2):790--813, 2009.

\bibitem{weakStab}
M.~D. Penrose and J.~E. Yukich.
\newblock Central limit theorems for some graphs in computational geometry.
\newblock {\em Ann. Appl. Probab.}, 11(4):1005--1041, 2001.

\bibitem{stab3}
M.~D. Penrose and J.~E. Yukich.
\newblock Weak laws of large numbers in geometric probability.
\newblock {\em Ann. Appl. Probab.}, 13(1):277--303, 2003.

\bibitem{timar2004tree}
{\'A}.~Tim{\'a}r et~al.
\newblock Tree and grid factors of general point processes.
\newblock {\em Electronic Communications in Probability}, 9:53--59, 2004.

\bibitem{WiRo70}
B.~Widom and J.~S. Rowlinson.
\newblock New model for the study of liquid--vapor phase transitions.
\newblock {\em J. Chem. Phys.}, 52(4):1670--1684, 1970.

\end{thebibliography}
\bibliographystyle{abbrv}

\appendix

\section{}
\label{appendix}

Let $\K$ be a Polish space and $\NK$ the space of boundedly finite Borel measures on $\R^d\times\K$ with the vague topology (cf~\cite[Section~4.1]{randMeas}). The law $\mc L(\Psi)$ of a random measure $\Psi$ on $\R\times\K$ is its probability distribution on $\NK$ with the respective Borel $\sigma$-algebra (equal to the corresponding evaluation $\sigma$-algebra).
A sequence of random measures $\Psi_n$ on $\R\times\K$ converges {\em (vaguely) in distribution} to $\Psi$, in symbols $\Psi_n{\buildrel vd\over\longrightarrow} \Psi$, if the sequence of distributions $\mc L(\Psi_n)$ 
weakly converges to $\mc L(\Psi)$ with respect to the vague topology on $\NK$.
 We shall prove the following result. Note, our statement~3 convergence of Palm probabilities is related to \cite[Theorem~10.4]{randMeas} but does not seem to follow from this result concerning non-stationary, non-marked point processes.

\begin{lemma}[Tightness and convergence of general stationary markings]\label{markStabLem-general}
Let $\mc T$ be the family of (distributions of) stationary marked point process on $\R^d\times\mathbb K$ 
with uniformly bounded intensities, \begin{equation}\label{e.bounded-intensity-lambda0}\sup_{\Psi \in\mc T}\E[\Psi(Q_1 \times \mathbb K)]<\infty.
\end{equation}
\begin{enumerate}
\item \label{i.lemma} If $\K$ is {locally compact}, then $\mc T$ is tight as a family of probability measures on~$\NK$. 
\item\label{ii.lemma} The family $\mc T$ is closed with respect to convergence in distribution.
\item\label{iii.lemma} Assume $\Psi_n{\buildrel vd\over\longrightarrow} \Psi$ 
for $\Psi_n\in\mc T$ and denote by $\lambda_n,\lambda$ the intensities of $\Psi_n$ and $\Psi$, respectively. If 
 $\lambda_n\longrightarrow\lambda$, then $\Psi_n^0{\buildrel vd\over\longrightarrow} \Psi^0$, where $\Psi_n^0,\Psi^0$ denote the respective Palm distributions.
\end{enumerate}
\end{lemma}

\begin{proof}
Set $\lambda_0 := \sup_{\Psi \in\mc T}\E[\Psi(Q_1\times \mathbb K)]$. Note that for any non-negative, bounded, boundedly supported Borel function $h$ on $\R^d\times\K$, any $M<\infty$, and 
any $\Psi\in\mc T$, we have
\begin{align*}
\P\left(\,\int_{\R^d\times\K}h\,\d\Psi> M\,\right)
&\le \frac1M \E\left[\int_{\R^d\times\K}h\,\d\Psi\right]\\
&\le \frac{\lambda_0}M \int_{\R^d}\E^0_\Psi[h(x,K_0)]\,\d x\\
&\le \frac{\lambda_0}M C\,||h||_{\infty},
\end{align*}
where $C<\infty$ is the Lebesgue measure of the projection of the support of $h$ on $\R^d$.
Thus, for given $h$, taking $M$ large enough one can upper-bound the probability $\P(\,\int h\,\d\Psi> M\,)$ uniformly for all $\Psi\in\mc T$. By~\cite[Proposition~11.1.VI]{pp1}
and the { local compactness of $\K$} (and hence of $\R^d\times\K$) we conclude that $\mc T$ is tight, thus proving statement~\eqref{i.lemma}.

Assume now $\Psi_n{\buildrel vd\over\longrightarrow} \Psi$.
The stationarity of $\Psi$ follows from the continuity of the shift operator $\shift_y$ with respect to the vague topology on $\NK$.
To prove that $\Psi$ has intensity bounded by $\lambda_0$, let $g$ be a 
continuous, non-negative, compactly supported Borel function on $\R^d$ and $g'_k$ continuous,
non-negative, compactly supported Borel functions on $\K$, such that $1_{K'_k}\le g'_k\le 1$, for some sequence of compact $K'_k\subset\K$ increasing to~$\K$. The convergence 
 $\Psi_n{\buildrel vd\over\longrightarrow} \Psi$ implies
 convergence in distribution 
 of the random variables 
 $\int gg'_k\,\d \Psi_n{\buildrel d\over\longrightarrow}_{n\to\infty}
 \int gg'_k\,\d \Psi$ (Theorem~\cite[Theorem~4.2~(ii)]{randMeas}) for all $k\ge1$. By Fatou's lemma
$$\E\Bigl[\int gg'_k\,\d \Psi\Bigr]\le \liminf_{n\to\infty}\E\Bigl[\int gg'_k\,\d \Psi_n\Bigr]\le
\lambda_0\int_{\R^d}g\,\d x,$$
where the last inequality is due to~\eqref{e.bounded-intensity-lambda0}.
Sending $k\to\infty$, by the monotone convergence theorem 
$$\E\Bigl[\int g\,\d \Psi\Bigr]\le \lambda_0\int_{\R^d}g\,\d x,$$
which proves that the intensity of $\Psi$ is bounded by $\lambda_0$ and completes the proof of statement~\eqref{ii.lemma}.

For the last statement, we first show that under the respective Palm distributions, the typical marks of $\Psi_n$ converge in distribution to those of $\Psi$.
In this regard consider continuous,
non-negative, compactly supported Borel functions $g, h$ on $\R^d$ and $\K$, respectively, with $\int_{\R^d}g(x)\,\d x=1$.
Denote by $\E^0_{n}$ and $\E^0$ the expectations with respect to Palm probabilities of $\Psi_n$ and $\Psi$, respectively. 
Then,
$$\E^0_n[h(K_0)]=\frac1{\lambda_n}
\E\Bigl[\int gh\,\d\Psi_n\Bigr]$$
and similarly for $\E^0[h(K_0)]$.
The assumption $\Psi_n{\buildrel vd\over\longrightarrow} \Psi$ implies
$\int gh\,\d\Psi_n{\buildrel d\over\longrightarrow}\int gh\,\d\Psi$. Since, additionally, $\lambda_n\to\lambda$, we have uniform integrability and therefore
$$\E\Bigl[\int gh\,\d\Psi_n\Bigr]\longrightarrow\E\Bigl[\int gh\,\d\Psi\Bigr].$$
This, together with the assumption $\lambda_n\longrightarrow \lambda$ proves
that $\E^0_n[h(K_0)]\longrightarrow\E^0[h(K_0)]$ thus completing the proof of the convergence in Palm distribution of the typical marks.

In order to complete the proof of statement~\eqref{iii.lemma} consider the universal marks of $\Psi_n$ and $\Psi$, i.e., the original realizations centered at the respective points. In symbols,
for any point process on $\R^d\times\K$ define the universal mark of its ground point $X_i\in\R^d$ by $\overline K_i:=\shift_{X_i}\Psi$.
The universal marks are elements of the Polish space $\overline\K:=\NK$ and the ground process $\{X_i\}$ of $\Psi$ marked by the universal marks form a point process $\overline\Psi$ 
on $\N_{\overline\K}$. Moreover, the Palm distribution of the typical universal mark,
$\overline K_0=\shift_0\Psi=\Psi$ coincides with the Palm distribution of $\Psi$. It is enough to observe now (see arguments below) that the {\em universal marking mapping}
 $\Psi\mapsto\overline\Psi$ from $\NK$ to $\N_{\overline\K}$ is continuous with respect the vague topology, which makes $\Psi_n{\buildrel vd\over\longrightarrow} \Psi$ imply $\overline\Psi_n{\buildrel vd\over\longrightarrow} \overline\Psi$, and use the previously proved convergence of the typical marks to conclude statement~\eqref{iii.lemma} of Lemma~\ref{markStabLem-general}.
 
 To justify the aforementioned continuity of the universal marking consider a continuous, compactly supported, real valued function $f$ on $\R^d\times\N_{\overline\K}$.
Note the function $h(x,\mu):=f(x,\shift_x\mu)$ ($x\in\R^d$, $\mu\in\N_{\overline\K}$) is also continuous and compactly supported and, as such, it is uniformly continuous
(say, with respect to the metric 
$d((x,\mu), (y,\nu)):=|x-y|+d_v(\mu,\nu)$, where $d_v$ is some metric on $\N_{\overline\K}$ inducing the vague topology cf.~\cite[Lemma~4.6]{randMeas}). 

Consider $\psi_n\to\psi$ vaguely in $\N_\K$.
For $\epsilon>0$, by the uniform continuity of $h$,
$|h(x,\shift_x\psi_n)-h(x,\shift_x\psi)|\le \epsilon$ for $n$ large enough, uniformly in $x\in\R^d$. Thus, for 
$\overline \psi_n$ and $\overline\psi$ resulting from the universal marking of $\psi_n$ and $\psi$, respectively, we have 
\begin{equation}\label{e.continuity-lemma-bounds}
\int_D (h(x,\psi)-\epsilon)\,\psi_n(\d x\times\K)\le \int f\,\d \overline \psi_n\le 
\int_D (h(x,\psi)+\epsilon)\,\psi_n(\d x\times\K),
\end{equation}
where $D$ is the projection of the support of~$f$ on $\R^d$.
Now, note that
$$\lim_{n \to \infty}\int h(x,\psi)\psi_n(\d x\times\K)= 
\int h(x,\psi)\psi(\d x\times\K) = \int f(x,\shift_x\psi)\psi(\d x\times\K) = \int f\,\d \overline\psi.$$
Hence, the $\lim\inf_{n}$ and $\lim\sup_{n}$ of the 
lower and upper bound in~\eqref{e.continuity-lemma-bounds} are equal, respectively, to 
$\int f\,\d \overline\psi+A\epsilon$ and 
$\int f\,\d \overline\psi-A\epsilon$, where
$A:=\lim\sup_n\psi_n(D\times\K)$ is finite since $\psi_n\to\psi$. This completes our proof of the continuity of the universal marking and Lemma~\ref{markStabLem-general}.
\end{proof}

\end{document}